\documentclass[a4paper,11pt]{amsart}

       \usepackage{palatino, verbatim}
        \usepackage[latin1]{inputenc}
        \usepackage[T1]{fontenc}
        \usepackage{amsthm, amsfonts}
        \usepackage{amsfonts}
        \usepackage{graphicx}
        \usepackage{amssymb}
        \usepackage{amsmath}
        \usepackage{latexsym}
        \usepackage[all]{xy}

        \newtheorem{thm}{Theorem}[section]
          \newtheorem{cor}[thm]{Corollary}
          \newtheorem{lem}[thm]{Lemma}
          \newtheorem{prop}[thm]{Proposition}

        \theoremstyle{definition}
          
          \newtheorem{rem}{Remark}

          \newcommand\M{{\mathcal M}}
           
          \newcommand\E{{\mathcal E}}
          \newcommand\N{{\mathcal N}}
           \newcommand\F{{\mathcal F}}
           \newcommand\G{{\mathcal G}}
          \newcommand\im{\mathrm{Im}}
          \newcommand\Pic{\mathrm{Pic}}
          \newcommand\W{\mathcal W}
          \newcommand\oo{\mathcal O}
          \newcommand\Z{\mathbb{Z}}
          \newcommand\Ext{\mathrm{Ext}}
          
          \newcommand\Hom{\mathrm{Hom}}
          \newcommand\Cliff{\mathrm{Cliff}}
          \newcommand\rk{\mathrm{rk}}
          \newcommand\NL{\mathcal{NL}}
           \newcommand\U{\mathcal U}
          \newcommand\K{\mathcal K}

\title
{Stability of rank-$3$ Lazarsfeld-Mukai bundles on $K3$ surfaces}
\author{Margherita Lelli--Chiesa}
\address{Humboldt Universit\"at zu Berlin, Institut f\"ur Mathematik, 10099 Berlin}
\email{lelli@math.hu-berlin.de}

\begin{document}
\begin{abstract}
Given an ample line bundle $L$ on a $K3$ surface $S$, we study the slope stability with respect to $L$ of rank-$3$ Lazarsfeld-Mukai bundles associated with complete, base point free nets of type $g^2_d$ on curves $C$ in the linear system $\vert L\vert$. When $d$ is large enough and $C$ is general, we obtain a dimensional statement for the variety $W^2_d(C)$. If the Brill-Noether number is negative, we prove that any $g^2_d$ on any smooth, irreducible curve in $\vert L\vert$ is contained in a $g^r_e$ which is induced from a line bundle on $S$, thus answering a conjecture of Donagi and Morrison. Applications towards transversality of Brill-Noether loci and higher rank Brill-Noether theory are then discussed.
\end{abstract}
\maketitle

\section{Introduction and statement of the results}
Many results of Brill-Noether theory regarding a general point in the moduli space $M_g$, which parametrizes isomorphism classes of smooth, irreducible curves of genus $g$, have been proved by studying curves lying on $K3$ surfaces. One of the advantages of considering an irreducible curve $C\subset S$, where $S$ is a smooth $K3$ surface, is that some interesting properties, such as the Clifford index, do not change while moving $C$ in its linear system (cf. \cite{green}). Moreover, Brill-Noether theory on $C$ is strictly connected with the geometry of some moduli spaces of vector bundles on the $K3$ surface.  Indeed, given a complete, base point free linear series $A$ on $C$, one associates with the pair $(C,A)$ a vector bundle on $S$,  the so-called Lazarsfeld-Mukai bundle, denoted by $E_{C,A}$.

Lazarsfeld-Mukai bundles were first used by Lazarsfeld, in order to show that, given a $K3$ surface $S$ such that $\Pic(S)=\Z\cdot L$, a general curve $C\in\vert L\vert$ satisfies the Gieseker-Petri Theorem, that is, for any line bundle $A\in \Pic(C)$ the Petri map
$$
\mu_{0,A}:H^0(C,A)\otimes H^0(C,\omega_C\otimes A^\vee)\to H^0(C,\omega_C)
$$
is injective (cf. \cite{lazarsfeld}, \cite{pareschi}, or \cite{lazarsfeld1} for a more geometric argument). 

It is natural to investigate what happens if the Picard number of $S$ is greater than $1$. In order to do this, having denoted by $\vert L\vert_s$ the locus of smooth, connected curves in the linear system $\vert L\vert$ and chosen two positive integers $r,d$, one studies the natural projection $\pi:\W^r_d(\vert L\vert)\to \vert L\vert_s$, whose fibre over $C$ coincides with the Brill-Noether variety $W^r_d(C)$. We set $g:=1+L^2/2$; this coincides with the genus of curves in $\vert L\vert_s$.

At first we look at the cases where $\rho(g,r,d)<0$. Following \cite{donagi}, we say that a line bundle $M$ is {\em adapted} to $\vert L\vert$ whenever 
\renewcommand{\theenumi}{\alph{enumi}}
\begin{enumerate}
\item[\em (i)]\label{pe2} $h^0(S,M)\ge2$, $h^0(S,L\otimes M^\vee)\geq 2$,
\item[\em (ii)]\label{pe4} $h^0(C,M\otimes\oo_{C})$ is independent of the curve $C\in\vert L\vert _s$.
\end{enumerate}
Conditions (i) and (ii) ensure that $M\otimes \oo_C$ contributes to the Clifford index of $C$ and $\Cliff(M\otimes \oo_C)$ is the same for any $C\in\vert L\vert_s$.

Donagi and Morrison (\cite{donagi} Theorem (5.1')) proved that, if $A$ is a complete, base point free pencil $g^1_d$  on a nonhyperelliptic curve $C\in\vert L\vert_s$ and $\rho(g,1,d)<0$, then $\vert A\vert$ is contained in the restriction to $C$ of a line bundle $M\in\Pic(S)$ which is adapted to $\vert L\vert$ and such that $\Cliff(M\otimes \oo_C)\leq \Cliff(A)$. The same is expected to hold true for any linear series of type $g^r_d$ with $\rho(g,r,d)<0$ (compare with \cite{donagi} Conjecture (1.2)). We prove this conjecture for $r=2$ under some mild hypotheses on $L$.
\begin{thm}\label{thm:magari}
Let $S$ be a $K3$ surface and $L\in\Pic(S)$ be an ample line bundle such that a general curve in $\vert L\vert$ has genus $g$, Clifford dimension $1$ and maximal gonality $k=\left\lfloor \frac{g+3}{2}\right\rfloor$. Let $A$ be a complete, base point free $g^2_d$ on a curve $C\in\vert L\vert_s$ such that $\rho(g,2,d)<0$. 

Then, there exists $M\in\Pic(S)$ adapted to $\vert L\vert$ such that the linear system $\vert A\vert$ is contained in $\vert M\otimes \oo_C\vert$ and $\Cliff(M\otimes \oo_C)\leq \Cliff(A)$. Moreover, one has $c_1(M)\cdot C\leq (4g-4)/3$.
\end{thm}
We recall that the assertion that $\vert A\vert$ is contained in $\vert M\otimes \oo_C\vert$ is equivalent to the requirement $h^0(C, A^\vee\otimes M\otimes \oo_C)>0$.
The assumption on the gonality $k$ is used for computational reasons; however, the methods of our proof might be adapted in order to treat the cases where $k$ is not maximal. It was proved by Ciliberto and Pareschi (cf. \cite{ciliberto} Proposition 3.3) that the ampleness of $L=\oo_S(C)$ forces $C$ to have Clifford dimension $1$ with only one exception  occurring for $g=10$.

The case of pencils is very particular, since it involves vector bundles of rank $2$. Donagi and Morrison used the fact that any non-simple, indecomposable Lazarsfeld-Mukai bundle of rank $2$ can be expressed as an extension of the image and the kernel of a nilpotent endomorphism, which both have rank $1$. Their proof cannot be adapted to linear series with $r>1$, corresponding to Lazarsfeld-Mukai bundles of rank at least $3$. Our techniques consist of showing that, under the hypotheses of Theorem \ref{thm:magari}, the rank-$3$ Lazarsfeld-Mukai bundle $E=E_{C,A}$ is given by an extension
$$
0\to N\to E\to E/N\to 0,$$
where $N\in\Pic(S)$ and $E/N$ is a $\mu_L$-stable, torsion free sheaf of rank $2$. When $E$ is $\mu_L$-unstable, the line bundle $N$ coincides with its maximal destabilizing sheaf and the determinant of $E/N$ plays the role of the line bundle $M$ in the statement. Something similar happens if $E$ is properly $\mu_L$-semistable. 

This suggests that the notion of stability might play a fundamental role in a general proof of the Donagi-Morrison Conjecture. \vspace{0.3cm}
 
Now, we turn our attention to the cases where $\rho(g,r,d)\geq 0$. In the course of their proof of Green's Conjecture for curves on arbitrary $K3$ surfaces, Aprodu and Farkas (cf. \cite{aprodu}) showed that, if $L$ is an ample line bundle on a $K3$ surface such that a general curve $C\in\vert L\vert$ has Clifford dimension $1$ and gonality $k$, given $d>g-k+2$, any dominating component of $\W^1_d(\vert L\vert)$ corresponds to simple Lazarsfeld-Mukai bundles. In particular, when the gonality is maximal this ensures that, if $C$ is general in its linear system and the Brill-Noether number $\rho(g,1,d)$ is positive, the variety $W^1_d(C)$ is reduced and of the expected dimension. In the case $\rho(g,1,d)=0$, one finds that $W^1_d(C)$ is $0$-dimensional, even though not necessarily reduced. 

It is natural to wonder to what extent such a result can be expected to hold for linear series of type $g^r_d$ with $r>1$. We prove the following theorem.
\begin{thm}\label{thm:principale}
Let $S$ be a $K3$ surface and $L\in\Pic(S)$ be an ample line bundle such that a general curve in $\vert L\vert$ has genus $g$, Clifford dimension $1$ and maximal gonality $k=\left\lfloor \frac{g+3}{2}\right\rfloor$. Fix a positive integer $d$ such that $\rho(g,2,d)\geq 0$ and assume $(g,d)\not\in\{(2,4),(4,5),(6,6),(10,9)\}$.
Then, the following hold:
\renewcommand{\theenumi}{\alph{enumi}}
\begin{enumerate}
\item\label{aa} If $d> \frac{3}{4}g+2$, no dominating component of $\mathcal{W}^2_d(\vert L\vert)$ corresponds to rank-$3$ Lazarsfeld-Mukai bundles which are not $\mu_L$-stable.
\item\label{bb} If $d\leq\frac{3}{4}g+2$, let $\W$ be a dominating component of $\W^2_d(\vert L\vert)$ that corresponds to Lazarsfeld-Mukai bundles which are not $\mu_L$-stable. Then, there exists $M\in\Pic(S)$ adapted to $\vert L\vert$ such that, for a general $(C,A)\in\W$, the linear system $\vert A\vert$ is contained in $\vert M\otimes \oo_C\vert$ and $\Cliff(M\otimes \oo_C)\leq \Cliff(A)$.  Moreover, $c_1(M)\cdot C\leq (4g-4)/3$.
\end{enumerate}
\end{thm}
Unlike case (\ref{aa}), case (\ref{bb}) does not exclude the existence of dominating components of $\W^2_d(\vert L\vert)$ which correspond to either $\mu_L$-stable or properly $\mu_L$-semistable Lazarsfeld-Mukai bundles. However, general points of such a component $\W$ give nets $g^2_d$, which are all contained in the restriction of the same line bundle $M\in \Pic(S)$ to curves in $\vert L\vert$. Furthermore, the Clifford index of $M\otimes \oo_C$ is the same for any $C\in\vert L\vert_s$ and does not exceed $d-4$. 

For a curve $C\in\vert L\vert_s$ and for a fixed value of $d$, we define the variety
$$
\widetilde{W}^2_d(C):=\{A\in W^2_d(C)\,\,\vert\,\, A\textrm{ is base point free}\},
$$
which is an open subscheme of $W^2_d(C)$, not necessarily dense. The following result is a direct consequence of Theorem \ref{thm:principale}.
\begin{cor}
Under the same hypotheses of Theorem \ref{thm:principale}, for a general $C\in\vert L\vert_s$ the following hold. 
\renewcommand{\theenumi}{\alph{enumi}}
\begin{enumerate}
\item If $d> \frac{3}{4}g+2$, the variety $\widetilde{W}^2_d(C)$ is reduced of the expected dimension $\rho(g,2,d)$.
\item If $d\leq\frac{3}{4}g+2$, let $W$ be an irreducible component of $\widetilde{W}^2_d(C)$ which either is non-reduced or has dimension greater than $\rho(g,2,d)$. Then, there exists an effective divisor $D\subset S$ such that $\oo_S(D)$ is adapted to $\vert L\vert$ and, for a general $A\in W$, the linear system $\vert A\vert$ is contained in $\vert \oo_C(D)\vert$ and $\Cliff(\oo_C(D))\leq \Cliff(A)$.
\end{enumerate}
\end{cor}
Aprodu and Farkas' result follows from a parameter count for spaces of Donagi-Morrison extensions corresponding to non-simple Lazarsfeld-Mukai bundles of rank $2$. The strategy used to prove Theorem \ref{thm:principale} consists, instead, of counting the number of moduli of $\mu_L$-unstable and properly $\mu_L$-semistable Lazarsfeld-Mukai bundles of rank $3$; this involves Artin stacks that parametrize the corresponding Harder-Narasimhan and Jordan-H\"older filtrations.\vspace{0.3cm}

The plan of the paper is as follows. Sections \ref{background1} and \ref{background2} give background information on Lazarsfeld-Mukai bundles and stability of sheaves on $K3$ surfaces. 

In Section \ref{bistrot} we present a different proof of Aprodu and Farkas' result and show that, if $\rho(g,1,d)>0$, the Lazarsfeld-Mukai bundles corresponding to general points of any dominating component of $\W^1_d(\vert L\vert)$ are not only simple, but even $\mu_L$-stable (Theorem \ref{thm:stable}). We introduce stacks of filtrations, studied for instance by Bridgeland in \cite{bridgeland} and Yoshioka in \cite{yoshioka}, and explain our parameter count in an easier case. The space of Lazarsfeld-Mukai bundles $E$, such that the bundles appearing in the Harder-Narasimhan filtration of $E$ have prescribed Mukai vectors, turns out to be an Artin stack, whose dimension can be computed by using some well known facts regarding morphisms between semistable sheaves. 

In Section \ref{tempo} we look at the different types of possible Harder-Narashiman and Jordan-H\"older filtrations of a rank-$3$ Lazarsfeld-Mukai bundle $E$ with $\det(E)=L$ and $c_2(E)=d$. If the determinants of both the subbundles $E_i$ and the quotient sheaves $E^j$, given by the filtration of $E$, have at least $2$ global sections, their restriction to a general curve $C\in\vert L\vert$ contributes to the Clifford index. This is used in order to bound from below the intersection products between the first Chern classes of the sheaves $E_i$ and $E^j$.

In Sections \ref{mare}, \ref{section:cola}, \ref{spiaggia} we estimate the number of moduli of pairs $(C,A)$ corresponding to rank-$3$ Lazarsfeld-Mukai bunldes which are not $\mu_L$-stable. The subdivision in three sections reflects the different methods necessary to treat various types of filtrations, depending on their length and on the rank of the sheaves $E_i$ and $E^j$. At the end of Section \ref{spiaggia} the proofs of both Theorem \ref{thm:magari}  and Theorem \ref{thm:principale} are given.

In Section \ref{bus}, an application towards transversality of Brill-Noether loci and Gieseker-Petri loci is presented. Recall that the Gieseker-Petri locus $GP_g$ consists, by definition, of curves inside $M_g$ that violate the Gieseker-Petri Theorem. For values of $r,d$ such that $\rho(g,r,d)\geq 0$, one defines the component of $GP_g$ of type $(r,d)$ as
$$GP^r_{g,d}:=\{[C]\in M_g\,\vert\,\exists\,(A,V)\in G^r_d(C)\textrm{ with }\ker\mu_{0,V}\neq 0\},$$
where $\mu_{0,V}$ is the Petri map. The subscheme
$$\widetilde{GP}^r_{g,d}:=\{[C]\in M_g\,\vert\,\exists\,A\in W^r_d(C)\setminus W^{r+1}_d(C)\textrm{ with }\ker\mu_{0,A}\neq 0\}$$
is open in $GP^r_{g,d}$ but not necessarily dense. We prove the following:
\begin{thm}\label{thm:tra}
Let $r\geq 3$, $g\geq 0$, $d\leq g-1$ be positive integers such that $\rho(g,r,d)<0$ and $d-2r+2\geq \lfloor(g+3)/2\rfloor$. If $r\geq 4$, assume $d^2>4(r-1)(g+r-2)$. For $r=3$, let $d^2>8g+1$. If $-1$ is not represented by the quadratic form
$$
Q(m,n)=(r-1)m^2+mnd+(g-1)n^2,\,\,\,m,n\in\mathbb{Z},
$$
then:
\renewcommand{\theenumi}{\alph{enumi}}
\begin{enumerate}
\item $M^r_{g,d}\not\subset M^1_{g,f}$   for $f<(g+2)/2$.
\item $M^r_{g,d}\not\subset \widetilde{GP}^1_{g,f}$    for $f\geq (g+2)/2$.
\item $M^r_{g,d}\not\subset M^2_{g,e}$    if $e<d-2r+5$ and $\rho(g,2,e)<0$. 
\item $M^r_{g,d}\not\subset \widetilde{GP}^2_{g,e}$   if $e<\min\left\{\frac{17}{24}g+\frac{23}{12}, d-2r+5\right\}$ and $\rho(g,2,e)\geq 0$. 
\end{enumerate}
\end{thm}

The assumption on the quadratic form $Q$ is a mild hypothesis. For instance, it is automatically satisfied when $r$ and $g$ are odd and $d$ is even.

In the last section we exhibit an application of our methods to higher rank Brill-Noether Theory. We give a negative answer to Question 4.2 in \cite{new}, which asks whether the second Clifford index $\Cliff_2(C)$, associated with rank-$2$ vector bundles on a curve $C$, equals $\Cliff(C)$ whenever $C$ is a Petri curve. We analyze what happens in genus $11$ and look at the Noether-Lefschetz divisor $\NL^4_{11,13}$, which consists of curves that lie on a $K3$ surface $S\subset \mathbb{P}^4$ with Picard number at least $2$; this coincides with the locus of curves $[C]\in M_{11}$ such that $\Cliff_2(C)<\Cliff(C)$ (cf. \cite{gan}). We prove the following:
\begin{thm}\label{thm:chisa}
A general curve $[C]\in \NL^4_{11,13}$ satisfies the Gieseker-Petri Theorem.
\end{thm}
In other words, the Gieseker-Petri divisor $GP_{11}$ and the Noether-Lefschetz divisor $\NL^4_{11,13}$ are transversal.
\vspace{0.3cm}\\
\textbf{Acknowledgements:} This paper is part of my Ph.D. thesis and I am grateful to my advisor Gavril Farkas for discussions. I would like to thank Marian Aprodu for an inspiring conversation had last February in Berlin. A special thank goes to Peter Newstead for giving me the opportunity of spending a productive period at the Newton Institute in Cambridge and suggesting to me the genus-$11$ problem and further applications to higher rank Brill-Noether theory.

\section{Lazarsfeld-Mukai bundles}\label{background1}
In this section we briefly recall the definition and the main properties of Lazarsfeld-Mukai bundles (LM bundles in the sequel) associated with complete, base point free linear series on curves lying on $K3$ surfaces. We refer to \cite{lazarsfeld}, \cite{lazarsfeld1}, \cite{pareschi} for the proofs. Let $S$ be a $K3$ surface and $C\subset S$ a smooth connected curve of genus $g$. Any base point free linear series $A\in W^r_d(C)\setminus W^{r+1}_d(C)$ can be considered as a globally generated sheaf on $S$; therefore, the evaluation map $\mathrm{ev}_{A,S}:H^0(C,A)\otimes\oo_S\to A$ is surjective and one defines the bundle $F_{C,A}$ to be its kernel, i.e., 
\begin{equation}\label{equation:prima}
0\to F_{C,A}\to H^0(C,A)\otimes\oo_S\to A\to 0.
\end{equation}
The LM bundle associated with the pair $(C,A)$ is, by definition, $E_{C,A}:=F_{C,A}^\vee$. By dualizing (\ref{equation:prima}), one finds that $E_{C,A}$ sits in the following short exact sequence:
\begin{equation}\label{equation:seconda}
0\to H^0(C,A)^\vee\otimes\oo_S\to E_{C,A}\to \omega_C\otimes A^\vee\to 0;
\end{equation}
in particular, $E_{C,A}$ is equipped with a ($r+1$)-dimensional subspace of sections. The following proposition summarizes the most important properties of $E_{C,A}$:
\begin{prop}\label{prop:basic}
If $E_{C,A}$ is the LM bundle corresponding to a base point free linear series $A\in W^r_d(C)\setminus W^{r+1}_d(C)$, then:
\begin{itemize}
\item $\rk\, E_{C,A}=r+1$.
\item $\det E_{C,A}=L$, where $C\in\vert L\vert$.
\item $c_2( E_{C,A})=d$.
\item The bundle $E_{C,A}$ is globally generated off the base locus of $\omega_C\otimes A^\vee$.
\item $h^0(S, E_{C,A})=h^0(C,A)+h^0(C,\omega_C\otimes A^\vee)=r+1+g-d+r$,\\ $h^1(S, E_{C,A})=h^2(S, E_{C,A})=0$.
\item $\chi(S, E_{C,A}\otimes  F_{C,A})=2(1-\rho(g,r,d))$.
\end{itemize}
\end{prop}
In particular, if $\rho(g,r,d)<0$, the LM bundle $E_{C,A}$ is non-simple.

Being a LM bundle is an open condition. Indeed, a vector bundle $E$ of rank $r+1$ is a LM bundle whenever $h^1(S,E)=h^2(S,E)=0$ and there exists $\Lambda\in G(r+1,H^0(S,E))$ such that the degeneracy locus of the evaluation map $ev_\Lambda:\Lambda\otimes\oo_S\to E$ is a smooth connected curve. 

Analogously, given $(C,A)$ as above, one defines a rank-$r$ vector bundle $M_A$ on $C$ as the kernel of $\mathrm{ev}_{A,C}:H^0(C,A)\otimes\oo_C\to A$. It turns out that $$H^0(C,M_A\otimes \omega_C\otimes A^\vee)=\ker\mu_{0,A}.$$ Similarly, by tensoring (\ref{equation:seconda}) by $F_{C,A}$ and taking cohomology, one shows that $$H^0(S,E_{C,A}\otimes F_{C,A})\simeq H^0(C,F_{C,A}\otimes \omega_C\otimes A^\vee).$$ Moreover, there is the following short exact sequence:
\begin{equation}\label{nuvole}
0\to \oo_C\to F_{C,A}\otimes \omega_C\otimes A^\vee\to M_A\otimes \omega_C\otimes A^\vee\to 0.
\end{equation}
Having denoted by $\pi: \W^r_d(\vert L\vert)\to \vert L\vert_s$ the natural projection and by $\mu_{1,A,S}$ the composition of the Gaussian map $\mu_{1,A}:\ker\mu_{0,A}\to H^0(C,\omega_C^2)$ with the transpose of the Kodaira-Spencer map $\delta_{C,S}^\vee:H^0(C,\omega_C^2)\to (T_C\vert L\vert)^\vee=H^1(C,\oo_C)$, one has that $\im(d\pi_{(C,A)})\subset \mathrm{Ann}(\im(\mu_{1,A,S}))$. Sard's Lemma applied to the projection $\pi$ implies that, if $\W\subset \W^r_d(\vert L\vert)$ is a dominating component and $C$ is general in its linear system, the sequence (\ref{nuvole}) is exact on the global sections for any $(C,A)\in(\pi\vert_\W)^{-1}(C)$ such that $A$ is base point free and $h^0(C,A)=r+1$; indeed, the coboundary map $H^0(C, M_A\otimes \omega_C\otimes A^\vee)\to H^1(C,\omega_C)$ coincides, up to a scalar factor, with $\mu_{1,A,S}$. As a consequence, the simplicity of $E_{C,A}$ is equivalent to the injectivity of $\mu_{0,A}$. In particular, if general points of $\W$ are complete, base point free linear series corresponding to simple LM bundles, the fiber $(\pi\vert_\W)^{-1}(C)$ over a general $C\in\vert L\vert_s$ is reduced of the expected dimension. Standard Brill-Noether theory implies that no component of $\W^r_d(\vert L\vert)$ is entirely contained in $\W^{r+1}_d(\vert L\vert)$. Therefore, the variety $W^r_d(C)$ is reduced of the expected dimension for a general $C\in\vert L\vert_s$ if no dominating component $\W$ of $\W^r_d(\vert L\vert)$ is of one of the following types:
\begin{enumerate}
\item For $(C,A)\in\W$ general, $A$ is complete, base point free and $E_{C,A}$ is non-simple.
\item\label{referendum}  For $(C,A)\in\W$ general, $A$ is not base point free and $\ker\mu_{0,A}\neq 0$.
\end{enumerate}
In order to exclude (\ref{referendum}), one can proceed by induction on $d$ because, if $B$ denotes the base locus of $A$ and $\ker\mu_{0,A}\neq 0$, then $\mu_{0,A(-B)}\neq 0$, too.

\section{Mumford stability for sheaves on $K3$ surfaces}\label{background2}
For later use, we recall some facts about coherent sheaves on smooth projective surfaces referring to \cite{lehn} and \cite{shatz} for most of the proofs. Let $S$ be a smooth, projective surface over $\mathbb{C}$ and $H$ an ample line bundle on it. Given a torsion free sheaf $E$ on $S$ of rank $r$, the $H$-slope of $E$ is defined as
$$
\mu_H(E)=\frac{c_1(E)\cdot c_1(H)}{r},
$$
and $E$ is called $\mu_H$-semistable (resp. $\mu_H$-stable) in the sense of Mumford-Takemoto if for any subsheaf $0\neq F\subset E$ with $\rk\, F<\rk\, E$, one has $\mu_H(F)\leq \mu_H(E)$ (resp. $\mu_H(F)< \mu_H(E)$). 
The Harder-Narasimhan filtration of $E$ (HN filtration in the sequel) is the unique filtration
$$0=E_0\subset E_1\subset\ldots\subset E_s=E,$$
such that $E^i:=E_i/E_{i-1}$ is a torsion free, $\mu_H$-semistable sheaf for $1\leq i\leq s$, and $\mu_H(E_{i+1}/E_i)<\mu_H(E_i/E_{i-1})$ for $1\leq i\leq s-1$. Such a filtration always exists. It can be easily checked that, if $E$ is a vector bundle, the sheaves $E_i$ are locally free; moreover,
$$
\mu_H(E_1)>\mu_H(E_2)>\ldots >\mu_H(E).$$
The sheaf $E_1$ is called the maximal destabilizing sheaf of $E$; the number $\mu_H(E_1)$ is the maximal slope of a proper subsheaf of $E$ and, among the subsheaves of $E$ of slope equal to $\mu_H(E_1)$, the sheaf $E_1$ has maximal rank. In particular, $E_1$ is $\mu_H$-semistable.

Now, we assume $E$ is $\mu_H$-semistable. A Jordan-H\"older filtration of $E$ (later on, JH filtration) is a filtration
$$0=JH_0(E)\subset JH_1(E)\subset\ldots\subset JH_s(E)=E,$$
such that all the factors $\mathrm{gr}_i(E):=JH_i(E)/JH_{i-1}(E)$ are torsion free, $\mu_H$-stable sheaves of slope equal to $\mu_H(E)$. This implies that $\mu_H(JH_i(E))=\mu_H(E)$ for $1\leq i\leq s$. The Jordan-H\"older filtration always exists but is not uniquely determined, while the graded object $\mathrm{gr}(E):=\oplus_i  \mathrm{gr}_i(E)$ is.                                                         

The following result concerns morphisms between $\mu_H$-semistable and $\mu_H$-stable sheaves on $S$ (cf. \cite{shatz}, \cite{friedman}).  
\begin{prop}\label{prop:morfismi}
Given two torsion free sheaves $E$ and $F$ on $S$, the following holds:
\begin{enumerate}
\item\label{utile} If $E$ and $F$ are $\mu_H$-semistable and $\mu_H(E)>\mu_H(F)$, then $\Hom(E,F)=0$.
\item If $E$ and $F$ are $\mu_H$-stable, $\mu_H(E)=\mu_H(F)$ and there exists $0\neq\varphi\in \Hom(E,F)$, then $\rk\, E=\rk\, F$ and $\varphi$ is an isomorphism in codim $\leq 1$ (in particular it is injective).
\end{enumerate}
\end{prop}
In the case where $S$ is a $K3$ surface, by Serre duality $H^2(S,E)\simeq\Hom(E,\oo_S)^\vee$; hence (\ref{utile}) implies that, if $E$ is $\mu_H$-semistable and $\mu_H(E)>0$, then $h^2(S,E)=0$.

From now on, we assume $S$ to be a $K3$ surface. Throughout the paper we will often use the following fact:
\begin{lem}\label{lem:marghe}
Let $E,Q\in\mathrm{Coh}(S)$ be torsion free and $\rk\, E\geq 2$. If $E$ is globally generated off a finite number of points, $h^2(S,E)=0$ and there exists a surjective morphism $\varphi:E\to Q$, then $h^0(S,Q^{\vee\vee})\geq 2$. 
\end{lem}
\begin{proof}
Being a quotient of $E$ , the sheaf $Q$ is globally generated off a finite set. If $\rk\, Q\geq 2$, this trivially implies $h^0(S,Q^{\vee\vee})\geq h^0(S,Q)\geq 2$. On the other hand, if $Q$ has rank $1$, then $Q=N\otimes I$, where $N\in\Pic(S)$ and $I$ is the ideal sheaf associated with a $0$-dimensional subscheme of $S$. Since $N$ is a quotient of $E$ off a finite number of points, it has no fixed components, thus it is base point free (cf. \cite{donat}). The statement follows by remarking that $N=Q^{\vee\vee}$ cannot be trivial because $h^2(S,E)=0$.
\end{proof}
Another useful result is the following one (cf. Lemma 3.1 in \cite{green}):
\begin{lem}\label{lem:evvai}
Let $E$ be a vector bundle of rank $r$ on $S$ which is globally generated off a finite number of points. If $h^2(S,E)=0$, then $h^0(S,\det E)\geq 2$.
\end{lem}
\begin{proof}
Since the natural map $\wedge^rH^0(S,E)\otimes\oo_S\to\wedge^rE=\det E$ is surjective off a finite number of points, the line bundle $\det E$ is base point free. Therefore, it is enough to show that $\det E$ is non-trivial. This follows by remarking that, given a general $V\in G(r,H^0(S,E))$, the natural map $ev_V:V\otimes \oo_S\to E$ is injective but is not an isomorphism since $h^2(S,E)=0$. Therefore, $\det ev_V$ gives a section of $\det E$ vanishing on a non-zero effective divisor.
\end{proof}
Last but not least, we recall some notation and results from \cite{tata}. The Mukai vector of a sheaf $E\in\mathrm{Coh}(S)$ is defined as:
$$
v(E):=\mathrm{ch}(E)(1+\omega)=\mathrm{rk}(E)+c_1(E)+(\chi(E)-\mathrm{rk}(E))\omega\in H^*(S,\Z)=H^{2*}(S,\mathbb{Z}),
$$
where $H^4(S,\Z)$ is identified with $\Z$ by means of the fundamental cocycle $\omega$. The Mukai lattice is the pair $(H^*(S,\Z),\langle,\rangle)$, with $\langle,\rangle$ being the symmetric bilinear form on $H^*(S,\Z)$ whose definition is the following:
$$
\langle v,w\rangle:=-\int_Sv^*\wedge w,
$$
where, if $v=v^0+v^1+v^2$ with $v^i\in H^{2i}(S,\Z)$, we set $v^*:=v^0-v^1+v^2$. Given $E,F\in\mathrm{Coh}(S)$, we define the Euler characteristic of the pair $(E,F)$ as
$$
\chi(E,F):=\sum_{i=0}^2(-1)^i \dim\Ext^i(E,F),
$$
and it turns out that $\chi(E,F)=-\langle v(E),v(F)\rangle$.

Given a Mukai vector $v\in H^*(S,\Z)$, let $\M(v)$ be the moduli stack of coherent sheaves on $S$ of Mukai vector $v$. If $H\in\Pic(S)$ is ample, we denote by $\M_H(v)^{\mu ss}$ (resp. $\M_H(v)^{\mu s}$) the moduli stack parametrizing isomorphism classes of $\mu_H$-semistable (resp. $\mu_H$-stable) sheaves on $S$ with Mukai vector $v$. Recall that any $\mu_H$-stable sheaf is simple and that any irreducible component of $\M_H(v)^{\mu s}$  has dimension equal to $\langle v,v\rangle +1$. Moreover, if $\gcd(v^0,v^1. H)=1$, then $\mu_H$-semistability and $\mu_H$-stability coincide.

\section{Stability of Lazarsfeld-Mukai bundles of rank $2$}\label{bistrot}
Let $S$ be a smooth, projective $K3$ surface and consider a line bundle $L\in\mathrm{Ample}(S)$ such that a general curve $C\in\vert L\vert_s$ has genus $g$, Clifford dimension $1$ and maximal gonality $k=\left\lfloor \frac{g+3}{2}\right\rfloor$. In this section we prove that, if $C$ is general in its linear system and $\rho(g,1,d)>0$, the LM bundle associated with a general complete, base point free $g^1_d$ on $C$ is $\mu_L$-stable.

Fix a rank-$2$ LM bundle $E=E_{C,A}$ corresponding to a complete, base point free pencil $A\in W^1_d(C)$ with $C\in\vert L\vert_s$; Proposition \ref{prop:basic} implies that $$v(E)=2+c_1(L)+(g-d+1)\omega.$$
We assume $E$ is not $\mu_L$-stable.  In the case where $E$ is $\mu_L$-unstable (resp. properly $\mu_L$-semistable) we consider its HN filtration (resp. JH filtration) $0\subset M\subset E$, which gives a short exact sequence
\begin{equation}\label{domenica}
0\to M\to E\to N\otimes I_\xi\to 0,
\end{equation}
where $M$ and $N$ are two line bundles such that $\mu_L(M)>\mu_L(E)=g-1>\mu_L(N)$ (resp.  $\mu_L(M)=\mu_L(E)=\mu_L(N)$) and $I_\xi$ is the ideal sheaf of a $0$-dimensional subscheme $\xi\subset S$ of length $l=d-c_1(N)\cdot c_1(M)$. By Lemma \ref{lem:marghe}, we know that $h^0(S,N)\geq 2$. First of all, we prove the following:
\begin{lem}\label{lem:uova}
In the situation above, if general curves in $\vert L\vert_s$ have Clifford dimension $1$ and (constant) gonality $k$, one has $c_1(M)\cdot c_1(N)\geq k$.
\end{lem}
\begin{proof}
We remark that $h^2(S,M)=0$ since $\mu_L(M)>0$. Therefore, if $$2>h^0(S,M)\geq\chi(M)=2+c_1(M)^2/2,$$ then $c_1(M)^2<0$ and the inequality $\mu_L(M)\geq g-1$ implies $c_1(M)\cdot c_1(N)\geq g+1\geq k$.\\ 
From now on, we assume $h^0(S,M)\geq 2$. Since $\omega_C\otimes N^\vee\vert_C=M\otimes\oo_C$, the line bundle $N\vert_C$ contributes to $\Cliff(C)$. The short exact sequence 
$$
0\to M^\vee\to N\to N\otimes\oo_C\to 0
$$
gives $h^0(C,N\otimes\oo_C)\geq h^0(S,N)$. It follows that 
\begin{eqnarray*}
\Cliff(N\otimes\oo_C)&=&c_1(N)\cdot (c_1(N)+c_1(M))-2h^0(C,N\otimes\oo_C)+2\\
&\leq& c_1(N)^2+c_1(N)\cdot c_1(M)-2\chi(N)-2h^1(S,N)+2\\
&=&-2+c_1(N)\cdot c_1(M)-2h^1(S,N).
\end{eqnarray*}
Since $\Cliff(N\otimes\oo_C)\geq k-2$, then $c_1(M)\cdot c_1(N)\geq k+2h^1(S,N)\geq k$.
\end{proof}
Our goal is to count the number of moduli of $\mu_L$-unstable and properly $\mu_L$-semistable LM bundles of rank $2$. 

Fix a nonnegative integer $l$ and a non-trivial, globally generated line bundle $N$ on $S$ such that, having defined $M:=L\otimes N^\vee$, either $\mu_L(M)=\mu_L(N)=g-1$ or $\mu_L(M)>g-1>\mu_L(N)$. We consider the moduli stack $\mathcal{E}_{N,l}$ parametrizing filtrations $0\subset M\subset E$ with $[M]\in\M(v(M))(\mathbb{C})$ and $[E/M]\in\M(v(N\otimes I_\xi))(\mathbb{C})$, where $l(\xi)=l$. Note that, since both $N$ and $M$ are line bundles, the stack $\M(v(M))$ has a unique $\mathbb{C}$-point endowed with an automorphism group of dimension $1$, while $\M(v(N\otimes I_\xi))$ is corepresented by the Hilbert scheme $S^{[l]}$ parametrizing $0$-dimensional subschemes of $S$ of length $l$. Two filtrations $0\subset M\subset E$ and $0\subset M'\subset E'$ are equivalent whenever there exists a commutative diagram
$$
\xymatrix{
M\ar[r]\ar[d]_{\varphi_1}&E\ar[d]^{\varphi_2}\\
M'\ar[r]&E',\\
}
$$
where $\varphi_1$ and $\varphi_2$ are two isomorphisms (cf. \cite{bridgeland} for the proof that $\mathcal{E}_{N,l}$ is algebraic). The stack $\mathcal{E}_{N,l}$ can be alternatively described as the moduli stack of extensions of type (\ref{domenica}). Let $p:\E_{N,l}\to \M(v(M))\times \M(v(N\otimes I_\xi))$ be the natural morphism of stacks mapping the short exact sequence (\ref{domenica}) to $(M,N\otimes I_\xi)$. The fiber of $p$ over the $\mathbb{C}$-point $(M,N\otimes I_\xi)$ of $\M(v(M))\times \M(v(N\otimes I_\xi))$ is the quotient stack 
$$
[\Ext^1(N\otimes I_\xi,M)/\Hom(N\otimes I_\xi,M)],
$$
where the action of $\Hom(N\otimes I_\xi,M)$ over $\Ext^1(N\otimes I_\xi,M)$ is the trivial one (cf. \cite{bridgeland}); it follows that in general $p$ is not representable.

We define $\tilde{P}_{N,l}$ to be the closure of the image of $\E_{N,l}$ under the natural projection $q:\E_{N,L}\to\M(v(E))$, which maps the point of $\E_{N,L}$ given by (\ref{domenica}) to $[E]$. The morphism $q$ is representable (cf. proof of Lemma (4.1) in \cite{bridgeland}) and the fiber of $q$ over a $\mathbb{C}$-point of $\tilde{P}_{N,l}$ corresponding to $E$ is the Quot-scheme $\mathrm{Quot}_S(E,P)$, where $\mathrm{P}$ is the Hilbert polynomial of $N\otimes I_\xi$. We denote by $P_{N,l}$ the open substack of $\tilde{P}_{N,l}$ whose $\mathbb{C}$-points correspond to vector bundles $E$ satisfying $h^1(S,E)=h^2(S,E)=0$.

Let $\G_{N,l}\to P_{N,l}$ be the Grassmann bundle with fiber over a point $[E]\in P_{N,l}(\mathbb{C})$ equal to $G(2,H^0(S,E))$. A $\mathbb{C}$-point of $\G_{N,l}$ is a pair $(E,\Lambda)$ and comes endowed with an automorphism group equal to $\mathrm{Aut}(E)$. We consider the rational map $$h_{N,l}:\G_{N,l}\dashrightarrow \W^1_d(\vert L\vert ),$$ mapping a general point $(E,\Lambda)\in\G_{N,l}(\mathbb{C})$ to the pair $(C_\Lambda,A_\Lambda)$, where $C_\Lambda$ is the degeneracy locus of the evaluation map $ev_\Lambda:\Lambda\otimes\oo_S\to E$, which is injective for a general $\Lambda\in G(2,H^0(S,E))$, and $\omega_{C_\Lambda}\otimes A_\Lambda^\vee$ is the cokernel of $ev_\Lambda$. Notice that $d:=c_1(N)\cdot c_1(M)+l$. Since while mapping to $\W^1_d(\vert L\vert )$ we forget the automorphisms, the fiber of $h_{N,l}$ over $(C,A)$ is the quotient stack
$$
[\mathbb{P}(\Hom(E_{C,A},\omega_C\otimes A^\vee)^\circ)/\mathrm{Aut}(E_{C,A})],
$$
where $\Hom(E_{C,A},\omega_C\otimes A^\vee)^\circ$ denotes the open subgroup of $\Hom(E_{C,A},\omega_C\otimes A^\vee)$ consisting of those morphisms whose kernel is isomorphic to $\oo_S^{\oplus 2}$, and $\mathrm{Aut}(E_{C,A})$ acts on $\mathbb{P}(\Hom(E_{C,A},\omega_C\otimes A^\vee)^\circ)$ by composition. In particular, $h_{N,l}$ is not representable. As remarked in Section \ref{background1}, one has $$\Hom (E_{C,A},\omega_C\otimes A^\vee)\simeq H^0(S,E_{C,A}\otimes E_{C,A}^\vee);$$ 
it is trivial to check that
$$
\Hom(E_{C,A},\omega_C\otimes A^\vee)^\circ\simeq \mathrm{Aut}(E_{C,A}).
$$
Therefore, the action of $\mathrm{Aut}(E_{C,A})$ on $\mathbb{P}(\Hom(E_{C,A},\omega_C\otimes A^\vee)^\circ)$ is transitive and the stabilizer of any point is the subgroup generated by $\mathrm{Id}_{E_{C,A}}$; as a consequence, any fiber of $h_{N,l}$ has dimension $-1$ (cf. \cite{gomez} for the definition of the dimension of a locally Noetherian algebraic stack). We denote by $\W_{N,l}$ the closure of the image of $h_{N,l}$. The following holds:
\begin{prop}\label{prop:no}
Assume that $P_{N,l}$ be non-empty and let $\W$ be an irreducible component of $\W_{N,l}$. Then
$$
\dim\W\leq g+d-k,
$$
where $k$ is the gonality of any curve in $\vert L\vert_s$.
\end{prop}
\begin{proof}
Proposition \ref{prop:morfismi}, together with the fact that $h^0(S,I_\xi)= 0$ if $l>0$, implies that
\begin{equation*}
\dim\Hom(M,N\otimes I_\xi)=\left\{\begin{array}{ll}1&\textrm{if }M\simeq N,\,\xi=\emptyset\\0&\textrm{otherwise}\end{array}\right..
\end{equation*}
It follows that the dimension of the fibers of $p$ is constant and equals $-\chi(M,N\otimes I_\xi)$, unless $M\simeq N$ and $l=0$, in which case it is $-\chi(M,N\otimes I_\xi)+1$. 

Regarding the fibers of $q$, it is well known (cf. \cite{lehn} Proposition 2.2.8) that, given $[\varphi:E\to N\otimes I_\xi]\in\mathrm{Quot}_S(E,P)$, the following holds:
\begin{equation}\label{quot}\begin{array}{lll}
\dim\Hom(K,N\otimes I_\xi)-\dim\Ext^1(K,N\otimes I_\xi)&\leq& \dim_{[\xi]}\mathrm{Quot}_S(E,P)\\&\leq&\dim\Hom(K,N\otimes I_\xi),
\end{array}\end{equation}
where $K=\ker\varphi$; moreover, if $\Ext^1(K,N\otimes I_\xi)=0$, then $\mathrm{Quot}_S(E,P)$ is smooth in $[\varphi]$ of dimension equal to $\dim\Hom(K,N\otimes I_\xi)$. Since $K\simeq M$, if $M\simeq N$ and $l=0$, the fibers of $q$ are smooth of dimension $1$; indeed, $\Ext^1(N,N)\simeq H^1(S,\oo_S)=0$. Otherwise, the fibers of $q$ are $0$-dimensional. It follows that, if $P_{N,l}$ is non-empty, then:
\begin{eqnarray*}
\dim \G_{N,l}&=&\dim P_{N,l}+2(g-d+1)\\
&=&\dim\M(v(M))+\dim\M(v(N\otimes I_\xi))+\langle v(M),v(N\otimes I_\xi)\rangle+2(g-d+1)\\
&=&2l-2+c_1(M)\cdot c_1(N)-\frac{c_1(M)^2}{2}-\frac{c_1(N)^2}{2}-2+l+2(g-d+1)\\
&=&3l+2g-2d-2-(g-1)+2c_1(M)\cdot c_1(N)\\
&=&g+d-1-c_1(N)\cdot c_1(M)\\
&\leq&g+d-1-k,
\end{eqnarray*}
where we have used that $c_1(M)+c_1(N)=c_1(L)$ and $d=c_1(M)\cdot c_1(N)+l$, and the last inequality follows from Lemma \ref{lem:uova}. The statement is a consequence of the fact that the fibers of $h_{N,l}$ are quotient stacks of dimension equal to $-1$.
\end{proof}
We can finally prove the following result:
\begin{thm}\label{thm:stable}
Assume that general curves in $\vert L\vert_s$ have Clifford dimension $1$ and maximal gonality $k=\left\lfloor \frac{g+3}{2}\right\rfloor$. 
\begin{itemize}
\item If $\rho(g,1,d)>0$, any dominating component of $\W^1_d(\vert L\vert)$ corresponds to $\mu_L$-stable LM bundles. In particular, if $C\in\vert L\vert_s$ is general, the variety $W^1_d(C)$ is reduced and has the expected dimension $\rho(g,1,d)$. 
\item If $\rho(g,1,k)=0$ and $C\in\vert L\vert_s$ is general, then $W^1_k(C)$ has dimension $0$. 
\end{itemize}
\end{thm}
\begin{proof}
When $\rho(g,1,d)>0$, we show that no component $\W$ of $\W^1_d(\vert L\vert_s)$ corresponding to either $\mu_L$-unstable or properly $\mu_L$-semistable LM bundles dominates $\vert L\vert$. Proposition \ref{prop:no} gives:
$$
\dim\W\leq g+d-k\leq g+d-\frac{g+2}{2}.
$$
Our claim follows by remarking that any dominating component of $\W^1_d(\vert L\vert)$ has dimension at least $g+\rho(g,1,d)$ and that $\rho(g,1,d)>d-\frac{g+2}{2}$ whenever $d>\frac{g+2}{2}$. 

If $k=\frac{g+2}{2}$, that is, $\rho(g,1,k)=0$, our parameter count shows that any dominating component of $\W^1_k(\vert L\vert)$ has dimension $g$;   hence, if $C\in\vert L\vert$ is general, $W^1_k(C)$ is $0$-dimensional, even though not necessarily reduced. By induction on $d$, one excludes the existence of components of $\W^1_d(\vert L\vert)$ whose general points correspond to linear series which are not base point free.
\end{proof}

\section{Lazarsfeld-Mukai bundles of rank $3$ which are not $\mu_L$-stable}\label{tempo}
We fix a LM bundle $E=E_{C,A}$ associated with a complete, base point free $g^2_d$ on a smooth connected curve $C\in \vert L\vert_s$ with $L\in\mathrm{Ample}(S)$. By Proposition \ref{prop:basic}, we have $$v(E)=3+c_1(L)+(2+g-d)\omega,$$ 
where $g=g(C)$. We assume that $E$ is not $\mu_L$-stable and, in the case where it is $\mu_L$-unstable, we look at its HN filtration:  
$$0=E_0\subset E_1\subset\ldots\subset E_s=E.$$
On the other hand, if $E$ is properly $\mu_L$-semistable, we consider its JH filtration:
$$0=JH_0(E)\subset JH_1(E)\subset\ldots\subset JH_s(E)=E.$$

We first consider the cases where either $E$ is properly $\mu_L$-semistable and $JH_1(E)$ has rank $2$, or $E$ is $\mu_L$-unstable, $\mathrm{rk}\,E_1=2$ and $E_1$ is $\mu_L$-stable. Under these hypotheses, $E$ sits in the following short exact sequence:
\begin{equation}\label{rango2}
0\to M\to E\to N\otimes I_\xi\to 0,
\end{equation}
where $M=JH_1(E)$ (resp. $M=E_1$) is a $\mu_L$-stable vector bundle of rank $2$, $N$ is a line bundle and $I_\xi$ is the ideal sheaf of a $0$-dimensional subscheme $\xi\subset S$. Moreover, 
\begin{equation}\label{bank}
\mu_L(M)\geq\mu_L(E)=\frac{2g-2}{3}\geq\mu_L(N\otimes I_\xi)=\mu_L(N),
\end{equation}
with the former inequality being strict whenever the latter one is. We have that $c_1(L)=c_1(E)=c_1(M)+c_1(N)$ and $d=c_2(E)=c_1(N)\cdot c_1(M)+l(\xi)+c_2(M)$, where $l(\xi)$ denotes the length of $\xi$. We prove the following:
\begin{lem}\label{lem:det}
Assume a general curve $C\in\vert L\vert_s$ has Clifford dimension $1$ and gonality $k$. In the above situation, one has  $c_1(N)\cdot c_1(M)\geq k$ and 
\begin{equation}\label{nene}
d\geq \frac{3}{4}k+\frac{7}{6}+\frac{g}{3}.
\end{equation}
\end{lem}
\begin{proof}
As $E$ is globally generated off a finite number of points, $N$ is base point free and non-trivial, thus $h^0(S,N)\geq 2$ and $\mu_L(N)>0$. The inequality $\mu_L(M)>0$ implies that $h^2(S,M)=0$ and, since $\mu_L(\det M)=2\mu_L(M)$, we have that $h^2(S,\det M)=0$, too. Therefore, $h^0(S,\det M)\geq\chi(\det M)=2+c_1(M)^2/2$ and, if $h^0(S,\det M)< 2$, then $c_1(M)^2\leq-2$ and $c_1(N)\cdot c_1(M)\geq(4g+2)/3>k$ by the first inequality in (\ref{bank}), which gives
\begin{equation}
c_1(M)^2+c_1(N)\cdot c_1(M)\geq\frac{4g-4}{3}.
\end{equation}
On the other hand, if $h^0(S,\det M)\geq 2$, then $N\vert_C$ contributes to $\Cliff(C)$ and one shows, as in the proof of Lemma \ref{lem:uova}, that $c_1(N)\cdot c_1(M)\geq k+2h^1(S,N)\geq k$. 

The $\mu_L$-stability of $M$ implies that
$$
-2\leq\langle v(M),v(M)\rangle=c_1(M)^2-4\chi(M)+8=4c_2(M)-c_1(M)^2-8.
$$
Therefore, we have
\begin{equation*}
d= c_1(N)\cdot c_1(M)+c_2(M)+l(\xi)\geq c_1(N)\cdot c_1(M)+\frac{c_1(M)^2}{4}+\frac{6}{4}\geq  \frac{3}{4}k+\frac{7}{6}+\frac{g}{3};
\end{equation*}
this concludes the proof.
\end{proof}
\vspace{0.5cm}

Now, we assume that either $E$ is $\mu_L$-unstable, $\mathrm{rk}\,E_1=1$ and $E/E_1$ is $\mu_L$-stable, or $E$ is properly $\mu_L$-semistable and its JH filtration is of type $0\subset JH_1(E)\subset E$ with $\mathrm{rk}\,JH_1(E)=1$. Denoting by $N$ the line bundle $E_1$ (resp. $JH_1(E)$), one has a short exact sequence:
\begin{equation}\label{rango1}
0\to N\to E\to E/N\to 0,
\end{equation}
where $E/N$ is a rank-$2$, $\mu_L$-stable, torsion free sheaf on $S$ such that $$\mu_L(N)\geq\mu_L(E)\geq\mu_L(E/N),$$ and either both inequalities are strict, or none is. We prove the following:
\begin{lem}\label{lem:endo1}
In the above situation, if a general curve $C\in\vert L\vert_s$ has Clifford dimension $1$ and gonality $k$, then $c_1(N)\cdot c_1(E/N)\geq k$.
\end{lem}
\begin{proof}
As in the proof of Lemma \ref{lem:marghe} one shows that $h^0(S,E/N)\geq 2$. Since $E/N$ is stable, then $\mu_L(E/N)>0$ and $h^2(S,E/N)=0$. Moreover, the vector bundle $(E/N)^{\vee\vee}$ is globally generated off a finite number of points and $h^0(S,\det (E/N))\geq 2$ by Lemma \ref{lem:evvai} because $\det (E/N):=\det (E/N)^{\vee\vee}$. 

Since $\mu_L(N)=c_1(N)\cdot (c_1(N)+c_1(E/N))\geq (2g-2)/3>0$, we have $h^2(S,N)=0$.\\ Hence, if $h^0(S,N)<2$, then $c_1(N)^2<0$ and $c_1(N)\cdot c_1(E/N)\geq (2g+4)/3> k$. Otherwise, $N\otimes\oo_C$ contributes to the Clifford index and this implies $c_1(N)\cdot c_1(E/N)\geq k$, too.
\end{proof}
\vspace{0.5cm}

The cases still to be considered are the following ones:
\renewcommand{\theenumi}{\roman{enumi}}
\begin{enumerate}
\item\label{one} $E$ is $\mu_L$-unstable with HN filtration $0\subset E_1\subset E_2\subset E$.
\item\label{two} $E$ is properly $\mu_L$-semistable with JH filtration $0\subset JH_1(E)\subset JH_2(E)\subset E$.
\item\label{three} $E$ is $\mu_L$-unstable with HN filtration $0\subset E_1\subset E$ and $E_1$ is a properly $\mu_L$- semistable vector bundle of rank $2$.
\item\label{four} $E$ is $\mu_L$-unstable with HN filtration $0\subset E_1\subset E$ and $E_1$ is a line bundle such that $E/E_1$ is a properly $\mu_L$- semistable torsion free sheaf of rank $2$.
\end{enumerate}
In all these cases one has four short exact sequences:

\begin{equation}\label{pizza}0\to N\to E\to E/N\to 0\end{equation}
\begin{equation}\label{arr}0\to M\to E\to N_1\otimes I_{\xi_1}\to 0,\end{equation}
\begin{equation}\label{arr1}0\to N\to M\to N_2\otimes I_{\xi_2}\to 0,\end{equation}
\begin{equation}\label{rango1b}0\to N_2\otimes I_{\xi_2}\to E/N\to N_1\otimes I_{\xi_1}\to 0,\end{equation}
where $N$, $N_1$, $N_2$ are line bundles, $I_{\xi_1}$ and $I_{\xi_2}$ denote the ideal sheaves of two $0$-dimensional subschemes $\xi_1,\xi_2\subset S$, the sheaf $E/N$ has rank-$2$ and no torsion, while $M$ is a vector bundle of rank $2$. 
Moreover, the following inequalities hold:
\begin{equation}\label{lun}
\mu_L(N)\geq\mu_L(N_2)\geq\mu_L(N_1),
\end{equation}  
\begin{equation}\label{mar}
\mu_L(N)\geq\frac{2g-2}{3}\geq\mu_L(N_1);
\end{equation} 
in particular, $\mu_L(N)=\mu_L(N_2)$ (resp. $\mu_L(N_1)=\mu_L(N_2)$) whenever $M$ (resp. $E/N$) is properly $\mu_L$-semistable, that is, in cases (\ref{two}) and (\ref{three}) (resp. in cases (\ref{two}) and (\ref{four})). Analogously, equalities in (\ref{mar}) force $E$ to be properly $\mu_L$-semistable with JH-filtration $0\subset N\subset M\subset E$, that is, one is in case (\ref{two}).

\begin{lem}\label{lem:caso3}
In the above situation, $N_1\otimes\oo_C$ always contributes to the Clifford index of $C\in\vert L\vert_s$. Moreover, one of the following occurs: 
\renewcommand{\theenumi}{\alph{enumi}}
\begin{enumerate}
\item\label{uno} Both $N\otimes\oo_C$ and $N_2\otimes\oo_C$ contribute to the Clifford index of $C\in\vert L\vert_s$.
\item\label{due} The inequality $c_1(N)\cdot (c_1(N_1)+c_1(N_2))\geq \frac{2g+4}{3}$ holds and either $N_2\otimes\oo_C$ contributes to the Clifford index of $C$ or $c_1(N_2)\cdot (c_1(N)+c_1(N_1))\geq g$.
\item\label{tre} The linear series $N\otimes\oo_C$ contributes to the Clifford index of $C\in\vert L\vert_s$ and one has the inequality $c_1(N_2)\cdot c_1(N)>\frac{1}{2}c_1(N)\cdot (c_1(N_1)+c_1(N_2))$.
\item\label{quattro} The inequality $c_1(N)\cdot c_1(N_2)\geq \frac{g+5}{3}$ holds.
\end{enumerate}
In particular, if a general $C\in\vert L\vert_s$ has Clifford dimension $1$ and gonality $k$, then
\begin{equation}\label{sese}
d\geq c_1(N)\cdot c_1(N_1)+c_1(N)\cdot c_1(N_2)+c_1(N_1)\cdot c_1(N_2)\geq\frac{3}{2}k.
\end{equation}
\end{lem}
\begin{proof}
Being a quotient of $E$ off a finite set, $N_1$ is base point free and non-trivial, thus $h^0(S,N_1)\geq 2$ and $\mu_L(N_1)>0$. By the "Strong Bertini' s Theorem" (cf. \cite{donat}), $N_1$ is nef. Proposition \ref{prop:morfismi} implies $h^2(S,N)=h^2(S,N_2)=0$ because of (\ref{lun}).
Analogously, $\mu_L(N_2\otimes N)=\mu_L(N_2)+\mu_L(N)>0$ and $h^2(S,N_2\otimes N)=0$. Moreover, the following holds:
\begin{eqnarray*}
c_1(N_2\otimes N)^2&=&c_1(N_2)^2+c_1(N)^2+2c_1(N_2)\cdot c_1(N)\\
&\geq&c_1(N)^2+c_1(N_2)\cdot c_1(N)+c_1(N_1)\cdot c_1(N)+c_1(N_1)^2\\
&=&\mu_L(N)+c_1(N_1)^2>0,
\end{eqnarray*}
where we have used that, since $\mu_L(N_2)\geq\mu_L(N_1)$, then 
\begin{equation}\label{banana}
c_1(N_2)^2+c_1(N_2)\cdot c_1(N)\geq c_1(N_1)^2+c_1(N_1)\cdot c_1(N),
\end{equation}
and that $c_1(N_1)^2\geq 0$ because $N_1$ is nef. We obtain that
$$
h^0(S, N_2\otimes N)\geq \chi(N_2\otimes N)=2+\frac{1}{2}c_1(N_2\otimes N)^2> 2,
$$
thus $N_1\otimes\oo_C$ always contributes to the Clifford index of $C\in\vert L\vert_s$.

If both $h^0(S,N_2)\geq 2$ and $h^0(S,N)\geq 2$, we are in case (\ref{uno}). 

If $h^0(S,N_2)\geq 2$ and $h^0(S,N)<2$, we show that (\ref{due}) occurs. Since $\chi(N)<2$, one has $c_1(N)^2<0$ and $c_1(N)\cdot (c_1(N_1)+c_1(N_2))\geq\mu_L(E)+2=(2g+4)/3$ by the first inequality in (\ref{mar}). Since $\mu_L(N\otimes N_1)>0$, then $h^2(S,N\otimes N_1)=0$. Moreover, one can show that 
$$
c_1(N\otimes N_1)^2\geq \mu_L(N_1)+c_1(N_2)^2>c_1(N_2)^2.
$$
It follows that, if $c_1(N\otimes N_1)^2<0$, then $c_1(N_2)^2<0$ and
$$
2g-2<2c_1(N)\cdot c_1(N_2)+2c_1(N_1)\cdot c_1(N_2),
$$
that is, $c_1(N_2)\cdot (c_1(N)+c_1(N_1))\geq g$. On the other hand, if $c_1(N\otimes N_1)^2\geq0$, then $h^0(S,N\otimes N_1)\geq 2$ and $N_2\otimes\oo_C$ contributes to the Clifford index.

From now on, assume $h^0(S,N_2)<2$, hence $c_1(N_2)^2<0$. Since $\det E/N\simeq N_1\otimes N_2$, Lemma \ref{lem:evvai} implies $h^0(S,N_1\otimes N_2)\geq 2$. Thus, if $h^0(S,N)\geq 2$, the linear series $N\otimes\oo_C$ contributes to the Clifford index of $C\in \vert L\vert_s$.
Furthermore, inequality (\ref{banana}), together with the fact that $c_1(N_2)^2<0\leq c_1(N_1)^2$, implies that $c_1(N_2)\cdot c_1(N)>c_1(N_1)\cdot c_1(N)$. We obtain $$c_1(N_2)\cdot c_1(N)>\frac{1}{2}c_1(N)\cdot (c_1(N_1)+c_1(N_2)),$$
and we are in case (\ref{tre})

It remains to treat the case where both $h^0(S,N_2)<2$ and $h^0(S,N)<2$. Under these hypotheses, $c_1(N_2)^2<0$ and $c_1(N)^2<0$ and we obtain
\begin{eqnarray*}
2g-2&\leq&c_1(N_1)^2+2c_1(N_1)\cdot c_1(N)+2c_1(N_1)\cdot c_1(N_2)+2c_1(N)\cdot c_1(N_2)-4\\
&=&2c_1(N)\cdot c_1(N_2)+2\mu_L(N_1)-c_1(N_1)^2-4\\
&\leq& 2c_1(N)\cdot c_1(N_2)+\frac{4g-4}{3}-4.
\end{eqnarray*}
As a consequence, $c_1(N)\cdot c_1(N_2)\geq\frac{g+5}{3}$ and we are in case (\ref{quattro}).

Now, we assume that $C$ has Clifford dimension $1$ and gonality $k$ and prove inequality (\ref{sese}). One shows, as in Lemma \ref{lem:uova}, that 
\begin{equation}\label{pilvia}
c_1(N_1)\cdot (c_1(N)+c_1(N_2))\geq k,
\end{equation}
because $N_1\otimes \oo_C$ always contributes to the Clifford index of $C\in\vert L\vert_s$.
Analogously, if $N\otimes \oo_C$ (resp. $N_2\otimes \oo_C$) contributes to $\Cliff(C)$, then $c_1(N)\cdot (c_1(N_1)+c_1(N_2))\geq k$ (resp. $c_1(N_2)\cdot (c_1(N)+c_1(N_1))\geq k$); therefore, the last part of the statement is proved if either (\ref{uno}) or (\ref{due}) occurs (use that $(2g+4)/3\geq k$). 

In case (\ref{tre}), one arrives at the same conclusion by adding inequality (\ref{pilvia}) and \begin{equation}\label{fra}c_1(N)\cdot c_1(N_2)>\frac{1}{2}c_1(N)\cdot (c_1(N_1)+c_1(N_2))\geq \frac{k}{2}.\end{equation} Similarly, in case (\ref{quattro}), one uses that $c_1(N)\cdot c_1(N_2)\geq(g+5)/3\geq k/2$.

\end{proof}

\begin{cor}\label{cor:lasagna}
Assume $C\in\vert L\vert_s$ has Clifford dimension $1$ and maximal gonality $k=\left\lfloor \frac{g+3}{2}\right\rfloor$ and let $E$ be the Lazarsfeld-Mukai bundle associated with a complete, base point free net $A\in W^2_d(C)$. If $E$ is not $\mu_L$-stable, $d<\frac{3}{4}k+\frac{7}{6}+\frac{g}{3}$ and $(g,d)\neq(6,6)$, then $E$ is given by an extension of type (\ref{rango1}), with $N\in\Pic(S)$ and $E/N$ a $\mu_L$-stable, torsion free sheaf of rank $2$ such that $\mu_L(N)\geq (2g-2)/3\geq\mu_L(E/N)$. \end{cor}
\begin{proof}
Apply Lemma \ref{lem:det} and Lemma \ref{lem:caso3} and remark that $\left\lceil\frac{3}{4}k+\frac{7}{6}+\frac{g}{3}\right\rceil\leq\left\lceil\frac{3}{2}k\right\rceil$ unless $g=6$.
\end{proof}

\section{Cases with a $\mu_L$-stable subbundle of rank $2$ and $L$-slope $\geq\mu_L(E)$ }\label{mare}
We assume that a general curve in $\vert L\vert$ has Clifford dimension $1$ and maximal gonality. In this section we show that, if $C\in\vert L\vert_s$ is general, the LM bundle $E$ corresponding to a general, complete, base point free $g^2_d$ on $C$ is neither properly $\mu_L$-semistable with JH filtration $0\subset JH_1(E)\subset E$ and $\mathrm{rk}\,JH_1(E)=2$, nor $\mu_L$-unstable with a $\mu_L$-stable, rank-$2$ vector bundle $E_1$ as maximal destabilizing sheaf .

Fix a positive integer $d$. Choose $l\in\mathbb{N}$ and a non-trivial, globally generated line bundle $N$ such that 
\begin{equation}\label{ale}
\mu_L(N)\leq\frac{2g-2}{3}\leq \frac{(c_1(L)-c_1(N))\cdot c_1(L)}{2},
\end{equation}
and impose that these are either two equalities or two strict inequalities.
Set 
\begin{eqnarray*}
c_1&:=&c_1(L)-c_1(N),\\
c_2&:=&d-c_1.c_1(N)-l,\\
\chi&:=&g-d+5-\chi(N)+l,
\end{eqnarray*}
 and define the vector $v:=2+c_1+(\chi-2)\omega\in H^*(S,\Z)$. The following construction is analogous to that of Section \ref{bistrot}.

Let $\mathcal{E}_{N,l}$ be the moduli stack of filtrations $0\subset M\subset E$, where $[M]\in \M_L(v)^{\mu s}(\mathbb{C})$ and $[E/M]\in \M(v(N\otimes I_\xi))(\mathbb{C})$ with $l(\xi)=l$. This is alternatively described as the moduli stack of extensions 
\begin{equation}\label{imp}
0\to M\to E\to N\otimes I_\xi\to 0,
\end{equation}
with $M$ and $\xi$ as above.

If $p:\E_{N,l}\to \M_L(v)^{\mu s}\times \M(v(N\otimes I_\xi))$ denotes the morphism of Artin stacks mapping the short exact sequence (\ref{imp}) to $(M,N\otimes I_\xi)$, the fiber of $p$ over the point of $M_L(v)^{\mu s}\times \M(v(N\otimes I_\xi))$ corresponding to the pair $(M,N\otimes I_\xi)$ is the quotient stack 
$$
[\Ext^1(N\otimes I_\xi,M)/\Hom(N\otimes I_\xi,M)].
$$
Define $\tilde{P}_{N,l}$ to be the closure of the image of $\E_{N,L}$ under the natural projection $$q:\E_{N,L}\to\M(v(E)),$$ which sends the isomorphism class of extension (\ref{imp}) to $[E]\in\M(v(E))(\mathbb{C})$. The morphism $q$ is representable and the fiber of $q$ over the point of $\tilde{P}_{N,l}$ corresponding to $[E]$ is the Quot-scheme $\mathrm{Quot}_S(E,P)$, where by $P$ we denote the Hilbert polynomial of $N\otimes I_\xi$.
We consider the open substack $P_{N,l}\subset\tilde{P}_{N,l}$, whose $\mathbb{C}$-points are isomorphism classes of vector bundles $E$ such that $h^1(S,E)=h^2(S,E)=0$.

\begin{lem}\label{lem:dim}
The stack $P_{N,l}$, if nonempty, has dimension
$$\dim P_{N,l}=2l+\langle v,v\rangle+\langle v(N\otimes I_\xi),v\rangle.
$$
\end{lem}
\begin{proof}
We claim that the dimension of the fibers of $p$ is constant. Indeed, Serre duality and Proposition \ref{prop:morfismi} imply that $\dim \Ext^2(N\otimes I_\xi,M)=\dim\Hom(M,N\otimes I_{\xi})=0$ for any $[M]\in \M_L(v)^{\mu s}(\mathbb{C})$ and $\xi\in S^{[l]}$.
This shows that $\E_{N,l}$, if nonempty, has dimension equal to
$$
\dim (M_L(v)^{\mu s}\times \M(v(N\otimes I_\xi)))-\chi(N\otimes I_\xi,M)=2l-1+1+\langle v,v\rangle+\langle v(N\otimes I_\xi),v\rangle;
$$
note that this coincides with the dimension computed by Yoshioka (cf. Lemma 5.2 in \cite{yoshioka}). The statement follows by remarking that, if $P_{N,l}$ is nonempty, then $\dim P_{N,l}=\dim\tilde{P}_{N,l}=\dim\E_{N,l}$ because the Quot-schemes corresponding to the fibers of $q$ are $0$-dimensional (use inequalities analogous to (\ref{quot})).

\end{proof}
We consider the Grassmann bundle $\G_{N,l}\to P_{N,l}$, whose fiber over $[E]\in P_{N,l}(\mathbb{C})$ is $G(3,H^0(S,E))$, and the rational map $h_{N,l}:\G_{N,l}\dashrightarrow \W^2_d(\vert L\vert )$. The fiber of $h_{N,l}$ over a pair $(C,A)$ is the quotient stack
$$
[\mathbb{P}(\Hom(E_{C,A},\omega_C\otimes A^\vee)^\circ)/\mathrm{Aut}(E_{C,A})],
$$
where $\Hom(E_{C,A},\omega_C\otimes A^\vee)^\circ\subset \Hom(E_{C,A},\omega_C\otimes A^\vee)$ consists, by definition, of morphisms with kernel isomorphic to $\oo_S^{\oplus 3}$. This quotient stack has dimension equal to $-1$, as in Section \ref{bistrot}. Our goal is to estimate the dimension of the closure of the image of $h_{N,l}$, which is denoted by $\W_{N,l}$. We first prove the following:
\begin{lem}\label{lem:occhio}
If $\G_{N,l}$ is nonempty, then
$$
\dim\G_{N,l}= g+\rho(g,2,d)+\chi(M,N\otimes I_\xi).
$$
Moreover, $\chi(M,N\otimes I_\xi)\leq\frac{4}{3}g+\frac{8}{3}-d-\frac{3}{2}c_1(N)\cdot c_1$.
\end{lem}
\begin{proof}
We use that
\begin{eqnarray*}
2(\rho(g,2,d)-1)=\langle v(E),v(E)\rangle&=&\langle v(N\otimes I_\xi),v(N\otimes I_\xi)\rangle+\langle v,v\rangle+2\langle v(N\otimes I_\xi),v\rangle\\
&=&2l-2+\langle v,v\rangle+2\langle v(N\otimes I_\xi),v\rangle;
\end{eqnarray*}
this implies that
\begin{eqnarray*}
\dim \G_{N,l}=\dim P_{N,l} +3(h^0(S,E)-3)&=&2\rho(g,2,d)-\langle v(N\otimes I_\xi),v\rangle+3(g-d+2)\\&=&g+\rho(g,2,d)+\chi(M,N\otimes I_\xi),
\end{eqnarray*}
as soon as $\G_{N,l}$ is nonempty. 

Since $\chi(M,N\otimes I_\xi)=-\langle v(N\otimes I_\xi),v\rangle=2\chi(N\otimes I_\xi)+\chi-4-c_1(N)\cdot c_1$, the last part of the statement follows by remembering that $\chi(E)=\chi+\chi(N\otimes I_\xi)=g-d+5$ and that
$$
\frac{c_1(N)^2}{2}\leq\frac{g-1}{3}-\frac{c_1(N)\cdot c_1}{2}
$$
because $\mu_L(E)\geq\mu_L(N\otimes I_\xi)$.
\end{proof}
In conclusion, we prove the following:
\begin{prop}\label{prop:cambridge}
Assume that a general curve in $\vert L\vert_s$ has Clifford dimension $1$ and maximal gonality $k=\left\lfloor \frac{g+3}{2}\right\rfloor$. Let $\W\subset \W_{N,l}$ be an irreducible component of $\W^2_d(\vert L\vert)$; then, $\rho(g,2,d)>0$ and $\W$ does not dominate the linear system $\vert L\vert$. 
\end{prop}
\begin{proof}
Lemma \ref{lem:det} gives $c_1(N)\cdot c_1\geq k\geq (g+2)/2$ and $d\geq \frac{3}{4}k+\frac{7}{6}+\frac{g}{3}\geq\frac{17}{24}g+\frac{23}{12}$; in particular, $\rho(g,2,d)> 0$. By Lemma \ref{lem:occhio}, we have
\begin{eqnarray*}
\dim\G_{N,l}&\leq& g+\rho(g,2,d)+\frac{4}{3}g+\frac{8}{3}-d-\frac{3}{2}k\\
&\leq& g+\rho(g,2,d)+\frac{4}{3}g+\frac{8}{3}-d-\frac{3}{4}g-\frac{3}{2}\\&=&g+\rho(g,2,d)+\frac{7}{12}g+\frac{7}{6}-d.
\end{eqnarray*}
Since any fiber of $h_{N,l}$ is an algebraic stack of dimension $-1$, then
$$
\dim\W\leq g+\rho(g,2,d)+\frac{7}{12}g+\frac{13}{6}-d.
$$
The right hand side is strictly smaller than $g+\rho(g,2,d)$ because $d> \frac{7g+26}{12}$. It follows that $\W$ cannot dominate $\vert L\vert$. 
\end{proof}

\section{Cases with a $\mu_L$-stable quotient sheaf of rank $2$ and $L$-slope $\leq\mu_L(E)$}\label{section:cola}
In this section we count the number of moduli of rank-$3$ LM bundles $E$, which are either properly $\mu_L$-semistable with JH filtration $0\subset JH_1(E)\subset E$ where $JH_1(E)$ is a line bundle, or $\mu_L$-unstable with maximal destabilizing sheaf $E_1$ such that $E/E_1$ is a $\mu_L$-stable, torsion free sheaf of rank $2$. 

Fix an integer $d\geq 4$. Choose $N\in\Pic(S)$ such that 
\begin{equation}\label{zia}
\mu_L(N)\geq\frac{2g-2}{3}\geq\frac{(c_1(L)-c_1(N))\cdot c_1(L)}{2},
\end{equation}
with equality holding either everywhere or nowhere. 

As before, we set $c_1':=c_1(L)-c_1(N)$, $c_2':=d-c_1'\cdot c_1(N)$, $\chi':=g-d+5-\chi(N)$, $v':=2+c_1'+(\chi'-2)\omega\in H^*(S,\Z)$. 

We denote by $\F_{N}$ the algebraic stack of extensions
\begin{equation}\label{sole}
0\to N\to E\to E/N\to 0,
\end{equation}
where $E/N$ defines a point of $\M_L^{\mu s}(v')$. Equivalently, $\F_{N}$ is the moduli stack of filtrations $0\subset N\subset E$ such that $[E/N]\in\M_L^{\mu s}(v')(\mathbb{C})$. Consider the two projections $p:\F_N\to \M_L^{\mu s}(v')\times \M(v(N))$ and $q:\F_N\to \M(v(E))$ and define $\tilde{R}_{N}$ to be the closure of the image of $q$. The open substack $R_{N}\subset\tilde{R}_{N}$ consists, by definition, of points corresponding to bundles $E$ such that $h^1(S,E)=h^2(S,E)=0$. We look at the Grassmann bundle $\G_N\to R_N$ with fiber over $[E]\in R_N(\mathbb{C})$ equal to $G(3,H^0(S,E))$. The closure of the image of $\G_N$ under the rational map $h_N:\G_N\dashrightarrow \W^2_d(\vert L\vert)$  is denoted by $\W_N$. As before, the fibers of $h_N$ are quotient stacks of dimension $-1$.
\begin{lem}\label{lem:fontana}
The stack $\G_N$, if nonempty, has dimension
$$
\dim\G_N= g+\rho(g,2,d)+\chi(E/N,N).
$$
\end{lem}
\begin{proof}
The fiber of $p$ over a point of $\M_L^{\mu s}(v')\times \M(v(N))$ corresponding to $(E/N,N)$ is the quotient stack $[\Ext^1(E/N,N)/\Hom(E/N,N)]$. Since $\mu_L(N)\geq\mu_L(E/N)$ and $E/N$ is $\mu_L$-stable, Serre duality and Proposition \ref{prop:morfismi} imply that $\Ext^2(E/N,N)=0$; hence, the dimension of the fibers of $p$ is constantly equal to $-\chi(E/N,N)=\langle v(N),v'\rangle$. The morphism $q$ is representable and, as in the previous sections, one shows that its fibers are Quot-schemes of dimension $0$. Therefore, if $R_N$ is nonempty, one has:
$$
\dim R_N=\dim\tilde{R}_N=\dim\F_N=\langle v',v'\rangle+\langle v(N),v'\rangle.
$$
The statement follows by proceeding as in the proof of Lemma \ref{lem:occhio}.
\end{proof}
The next Lemma gives an upper bound for $\chi(E/N,N)$.
\begin{lem}\label{lem:burraco}
Assume that a general curve $C\in\vert L\vert_s$ has Clifford dimension $1$ and maximal gonality $k=\left\lfloor \frac{g+3}{2}\right\rfloor$. If $R_N$ is nonempty, then $\chi(E/N,N)\leq \frac{3}{2}g-2d+3$ for any $E/N$ corresponding to a point of $\M_L^{\mu s}(v')$.
\end{lem}
\begin{proof}
Consider the extension (\ref{sole}), where $[E]\in R_N(\mathbb{C})$. Since $\mu_L(N)>0$, one has $h^1(S,E/N)=h^2(S,N)=0$. As in Lemma \ref{lem:marghe} one obtains $\chi(E/N)=h^0(S,E/N)\geq 2$, hence $\chi(N)=\chi(E)-\chi(E/N)\leq g-d+3$. As a consequence:
\begin{eqnarray*}
\chi(E/N,N)&=&2\chi(N)+\chi'-4-c_1(N)\cdot c_1'\\
&=&g-d+1+\chi(N)-c_1(N)\cdot c_1'\\
&\leq &2g-2d+4-c_1(N)\cdot c_1'\\
&\leq&\frac{3}{2}g-2d+3,
\end{eqnarray*}
where the last inequality follows from Lemma \ref{lem:endo1}.
\end{proof}
Finally, we prove the following:
\begin{prop}\label{prop:fiducia}
We assume that a general curve in $\vert L\vert$ has Clifford dimension $1$ and maximal gonality $k=\left\lfloor \frac{g+3}{2}\right\rfloor$. If $d>\frac{3}{4}g+2$, no irreducible component $\W$ of $\W^2_d(\vert L\vert)$ which is contained in $\W_N$ dominates the linear system $\vert L\vert$. 
\end{prop}
\begin{proof}
Let $\W\subset\W_N$ be an irreducible component of $\W^2_d(\vert L\vert)$. Since any fiber of $h_N$ is an Artin stack of dimension equal to $-1$, Lemma \ref{lem:fontana} and Lemma \ref{lem:burraco} imply that
$$
\dim\W\leq g+\rho(g,2,d)+\frac{3}{2}g-2d+4.$$
 
If $\rho(g,2,d)\geq 0$, the condition $d>\frac{3}{4}g+2$ prevents the map $\W\to \vert L\vert$ from being dominant.
\end{proof}
Now we show that, if $d$ is small enough and $C\in\vert L\vert_s$, any complete base point free $g^2_d$ on $C$, whose LM bundle is given by an extension of type (\ref{sole}), is contained in a linear series which is induced from a line bundle on $S$.
\begin{prop}\label{prop:vai}
Let $S$ and $L$ be as in the hypotheses of Proposition \ref{prop:fiducia} and $A$ be a complete, base point free $g^2_d$ on a curve $C\in\vert L\vert_s$. If $d<(5g+13)/6$ and the LM bundle $[E_{C,A}]\in R_N(\mathbb{C})$ for some $N\in\Pic(S)$, the linear system $\vert A\vert$ is contained in the restriction to $C$ of the linear system $\vert L\otimes N^\vee\vert$ on $S$. Moreover, $L\otimes N^\vee$ is adapted to $\vert L\vert$ and $\Cliff( L\otimes N^\vee\otimes\oo_C)\leq \Cliff(A)=d-4$.\end{prop}
\begin{proof}
By hypothesis, $E=E_{C,A}$ sits in a short exact sequence like (\ref{sole}), where $E/N$ is $\mu_L$-stable and $\mu_L(N)\geq (2g-2)/3\geq \mu_L(E/N)$. Since $\mu_L(N)>0$, then $h^2(S,N)=0$. 

The $\mu_L$-stability of $E/N$ implies 
$$-2\leq\langle v',v'\rangle=4c_2'-(c_1')^2-8,$$
thus $c_2'\geq 3/2+(c_1')^2/4$. 

If $h^0(S,N)<2$, then $c_1(N)^2\leq -2$, which implies $(c_1')^2+2c_1(N)\cdot c_1'\geq 2g$ and $c_1'\cdot c_1(N)\geq (2g+4)/3$. In particular, $$d=c_1'\cdot c_1(N)+c_2'\geq c_1'\cdot c_1(N)+\frac{3}{2}+\frac{(c_1')^2}{4}\geq \frac{g}{2}+\frac{3}{2}+\frac{g+2}{3}=\frac{5g+13}{6},$$ thus a contradiction. Therefore, one has both $h^0(S,N)\geq 2$ and $h^0(S,\det E/N)\geq 2$.

Remark that $(E/N)^{\vee\vee}$ is globally generated off a finite set and $$h^i(S,(E/N)^{\vee\vee})=h^i(S,E/N)=0\textrm{ for }i=1,2.$$ Since $\det E/N=\det (E/N)^{\vee\vee}$ is base point free and non trivial, if $h^1(S,\det E/N)\neq0$, then $(c_1')^2=0$ and  Proposition (1.1) in \cite{green} implies the existence of a smooth elliptic curve $\Sigma\subset S$ such that $$(E/N)^{\vee\vee}=\oo_S(\Sigma)\oplus\oo_S(\Sigma).$$ Such equality would contradict the stability of $E/N$, thus we conclude that $(c_1')^2\geq 2$ (and $c_2'\geq 2$) and \begin{equation}\label{treno}
h^1(S,\det E/N)=0.
\end{equation}
This ensures that $h^0(C,\det E/N\otimes \oo_C)$ does not depend on the curve $C\in \vert L\vert_s$ (cf. \cite{donagi} Lemma (5.2)). Hence, the line bundle $\det E/N=L\otimes N^\vee$ is adapted to $\vert L\vert$.

We obtain:
\begin{eqnarray*}
\Cliff(\det E/N\otimes\oo_C)&=&c_1(E/N)^2+c_1(N)\cdot c_1(E/N)-2h^0(C,\det E/N\otimes\oo_C)+2\\
&\leq&c_1(E/N)^2+c_1(N)\cdot c_1(E/N)-2h^0(S,\det E/N)+2\\
&=&c_1(N)\cdot c_1(E/N)-2-2h^1(S,\det E/N)\\
&=&d-c_2(E/N)-2\\
&\leq&d-4.
\end{eqnarray*}  

It remains only to prove that $h^0(C,\det E/N\otimes \oo_C\otimes A^\vee)>0$. Consider the following diagram:
$$
\xymatrix{
0\ar[r]&H^0(C,A)^\vee\otimes\oo_S\ar[r]&E\ar[r]^{\alpha}&\omega_C\otimes A^\vee\ar[r]&0.\\
&&N\ar@{^{(}->}[u]_{\gamma}&&
}
$$
Since $h^2(S,N)=0$, the composition $\alpha\circ\gamma\neq 0$. This implies $\Hom(N,\omega_C\otimes A^\vee)\neq 0$ and we have finished because $N^\vee\otimes\omega_C\otimes A^\vee\simeq\det E/N\otimes\oo_C\otimes A^\vee$.

\end{proof}

\section{Remaining cases}\label{spiaggia}
In this section we consider rank-$3$ LM bundles $E$ of type (\ref{one}), (\ref{two}), (\ref{three}), (\ref{four}) on page 12, such that $\det E=L$ and $c_2(E)=d$ is fixed. 

Choose $l_2\in\mathbb{N}$ and two line bundles $N,N_2\in\Pic(S)$ such that $N_1:=L\otimes(N\otimes N_2)^\vee$ is globally generated and non-trivial, and the following holds:
\begin{eqnarray}
\label{mami}\mu_L(N)\geq&\mu_L(N_2)\geq&\mu_L(N_1),\\
\label{pa}\mu_L(N)\geq&\frac{2g-2}{3}\geq&\mu_L(N_1),
\end{eqnarray}
where in (\ref{pa}) either both the inequalities are strict, or none is.

 Set $v:=v(N)$, $v_1:=v(N_1\otimes I_{\xi_1})$ and $v_2:=v(N_2\otimes I_{\xi_2})$, with $l(\xi_2)=l_2$ and $$l(\xi_1)=l_1:=d-l_2-c_1(N)\cdot c_1(N_1)-c_1(N)\cdot c_1(N_2)-c_1(N_1)\cdot c_1(N_2).$$

Define $\F_{N,N_2,l_2}$ to be the moduli stack of extensions 
$$0\to N_2\otimes I_{\xi_2}\to E/N\to N_1\otimes I_{\xi_1}\to 0,$$
where $\xi_i\subset S$ is a $0$-dimensional subscheme of length $l_i$  for $i=1,2$. We consider the projections $p_2:\F_{N,N_2,l_2}\to \M(v_2)\times \M(v_1)$ and $q_2:\F_{N,N_2,l_2}\to \M(v(E/N))$, and we denote by $Q_{N,N_2,l_2}$ the closure of the image of $q_2$. 

If $\E_{N,N_2,l_2}$ is the moduli stack of extensions
$$
0\to N\to E\to E/N\to 0,
$$
where $[E/N]\in Q_{N,N_2,l_2}(\mathbb{C})$, consider the morphisms $p_1:\E_{N,N_2,l_2}\to \M(v) \times Q_{N,N_2,l_2}$ and $q_1:\E_{N,N_2,l_2}\to \M(v(E))$. The closure of the image of $q_1$ is denoted by $\tilde{P}_{N,N_2,l_2}$ and its open substack, consisting of points which correspond to vector bundles $E$ such that $h^1(S,E)=h^2(S,E)=0$, by $P_{N,N_2,l_2}$.

Remark that, if $E$ is a LM bundle of type (\ref{one}), (\ref{two}), (\ref{three}) or (\ref{four}), there exist $N,N_2$ and $l_2$ such that $[E]$ defines a point of $P_{N,N_2,l_2}$. In order to count the number of moduli of such bundles, we start by proving the following:

\begin{lem}\label{lem:leo}
The stack $Q_{N,N_2,l_2}$, if nonempty, has dimension
$$
\dim Q_{N,N_2,l_2}=2l_1+2l_2-2+\langle v_1,v_2\rangle,
$$
unless $N_1\simeq N_2$, $l_2\neq 0$ and $l_1=0$. In this case, for any component $Q\subset Q_{N,N_2,l_2}$, the following inequality holds:
$$
\dim Q\leq 2l_1+2l_2-1+\langle v_1,v_2\rangle.
$$
\end{lem}
\begin{proof}
The fiber of $p_2$ over the point of $\M_L(v_2)\times\M_L(v_1)$ given by $(N_2\otimes I_{\xi_2},N_1\otimes I_{\xi_1})$ is the quotient stack 
$$
[\Ext^1(N_1\otimes I_{\xi_1},N_2\otimes I_{\xi_2})/\Hom(N_1\otimes I_{\xi_1},N_2\otimes I_{\xi_2})].
$$
Since $\mu_L(N_2)\geq\mu_L(N_1)$, if either $N_1\not\simeq N_2$ or $N_1\simeq N_2$, $l_1\neq0$ and $l_2= 0$, one finds that $$\Hom(N_2\otimes I_{\xi_2}, N_1\otimes I_{\xi_1})=0.$$ In these cases, the fibers of $p_2$ have constant dimension equal to $-\chi(N_1\otimes I_{\xi_1}, N_2\otimes I_{\xi_2})$ while the fibers of $q_2$ are $0$-dimensional Quot-schemes, hence the statement follows.

If $N_1\simeq N_2$ and $l_1=l_2= 0$, the conclusion is the same because the fibers of $p_2$ have constant dimension equal to $-\chi(N_1\otimes I_{\xi_1}, N_2\otimes I_{\xi_2})+1$ and the fibers of $q_2$ are smooth Quot-schemes of dimension $1$. Indeed, $\Hom (N,N)=1$ and $\Ext^1(N,N)=0$.

On the other hand, if $N_1\simeq N_2$ and $l_2\neq 0$, the fibers of $p_2$ do not necessarily have constant dimension; indeed, $\dim\Hom(N_1\otimes I_{\xi_2}, N_1\otimes I_{\xi_1})$ depends on the reciprocal position of $\xi_1$ and $\xi_2$. Since $\mathcal{H}om(I_{\xi_2},\oo_S)\simeq\mathcal{H}om(I_{\xi_2},I_{\xi_2})\simeq \oo_S$ (cf. \cite{okonek}), one shows that
$$
\mathcal{H}om(I_{\xi_2},I_{\xi_1})\simeq\{f\in\oo_S\,\vert\, f\cdot I_{\xi_2}\subseteq I_{\xi_1}\}=:(I_{\xi_1}:I_{\xi_2})=I_{\xi_1\setminus (\xi_1\cap \xi_2)};
$$
hence, one finds that
\begin{equation}
\dim\Hom(N_1\otimes I_{\xi_2}, N_1\otimes I_{\xi_1})=H^0(S,\mathcal{H}om(I_{\xi_2}, I_{\xi_1}))=\left\{\begin{array}{ll}1&\textrm{ if }\xi_1\subseteq \xi_2\\0&\textrm{ otherwise }\end{array}\right..
\end{equation}

As in \cite{yoshioka}, let $\N^0_{N,N_2,l_2}$ (resp. $\N^1_{N,N_2,l_2}$) be the substack of $\M(v_2)\times\M(v_1)$ whose points correspond to pairs $(N_1\otimes I_{\xi_2},N_1\otimes I_{\xi_1})$ such that $\xi_1\not\subseteq\xi_2$ (resp. $\xi_1\subseteq\xi_2$), that is,  $\dim\Hom(N_1\otimes I_{\xi_2}, N_1\otimes I_{\xi_1})=0$ (resp. $\dim\Hom(N_1\otimes I_{\xi_2}, N_1\otimes I_{\xi_1})=1$). Remark that $\N^0_{N,N_2,l_2}$ and $\N^1_{N,N_2,l_2}$ are complementary and that, being open, $\N^0_{N,N_2,l_2}$ is dense in $\M(v_2)\times\M(v_1)$ provided $l_1\neq0$. 

We define $\F_{N,N_2,l_2}^{0}:=(p_2)^{-1}(\N^0_{N,N_2,l_2})$ and $\F_{N,N_2,l_2}^{1}:=(p_2)^{-1}(\N^1_{N,N_2,l_2})$ and we denote by $Q_{N,N,l_2}^{0}$ and $Q_{N,N,l_2}^{1}$ the closures of the images under $q_2$ of $\F_{N,N_2,l_2}^{0}$ and $\F_{N,N_2,l_2}^{1}$ respectively. Since the fibers of $q_2$ are Quot-schemes, we obtain that:
\begin{eqnarray*}
\dim Q_{N,N_2,l_2}^{0}&=&\dim \F_{N,N_2,l_2}^{0}=\dim\N^0_{N,N_2,l_2}+\langle v_1,v_2\rangle\leq2l_1+2l_2-2+\langle v_1,v_2\rangle,\\
\dim Q_{N,N_2,l_2}^{1}&\leq&\dim \F_{N,N_2,l_2}^{1}=\dim\N^1_{N,N_2,l_2}+\langle v_1,v_2\rangle+1\leq2l_1+2l_2-1+\langle v_1,v_2\rangle,
\end{eqnarray*}
where the last inequality in the second row is strict, unless the stack $\N^1_{N,N_2,l_2}$ is dense in $\M(v_2)\times\M(v_1)$, that is, $l_1=0$.

The statement follows because every component of $Q_{N,N,l_2}$ is contained either in $Q_{N,N,l_2}^{0}$ or in $Q_{N,N,l_2}^{1}$.
\end{proof}
By proceeding as in Lemma \ref{lem:leo}, one proves the following:
\begin{prop}\label{prop:sedici}
Let $Z$ be a nonempty irreducible component of $P_{N,N_2,l_2}$. We have that
\begin{eqnarray}
\label{maggio}\dim Z&=& 2l_1+2l_2+\langle v_2,v\rangle+\langle v_1,v\rangle+\langle v_1,v_2\rangle-\alpha,
\end{eqnarray}
where $\alpha$ satisfies:
\renewcommand{\theenumi}{\alph{enumi}}
\begin{enumerate}
\item\label{unob} If $N$, $N_1$, $N_2$ are all non-isomorphic, then $\alpha=3$.
\item\label{dueb} Assume $N\simeq N_1\simeq N_2$. If $l_2\neq 0$ and $l_1=0$, then $\alpha\in\{1,2,3\}$. If $l_1\neq 0$ and $l_2=0$, one has $\alpha\in\{2,3\}$. In all the other cases, $\alpha=3$. If $N\simeq N_1\not\simeq N_2$, one has $\alpha=3$ unless $l_1=0$, in which case $\alpha\in\{2,3\}$.
\item\label{treb} If $N\simeq N_2\not\simeq N_1$, then $\alpha=3$ unless $l_2=0$, in which case $\alpha\in\{2,3\}$.
\item\label{quattrob} Assume $N_1\simeq N_2\not\simeq N$. Then $\alpha=3$ except when $l_2\neq0$ and $l_1= 0$; in this case $\alpha\in\{2,3\}$.
\end{enumerate}
\end{prop}
Note that LM bundles of type (\ref{one}) lie in some $P_{N,N_2,l_2}$ with $N,N_1$, $N_2$ as in case (\ref{unob}). Analogously, if $E$ is a LM bundle of type (\ref{three}) (resp. of type (\ref{four})), there exist $N,N_2,N_1=L\otimes(N\otimes N_2)^\vee$ as in (\ref{unob}) or (\ref{treb}) (resp. as in (\ref{unob}) or (\ref{quattrob})) and $l_2\in\mathbb{N}$ such that $[E]\in P_{N,N_2,l_2}(\mathbb{C})$. On the other hand, if a bundle of type (\ref{two}) defines a point of $P_{N,N_2,l}$, then $\mu_L(N)=\mu_L(N_2)=\mu_L(N_1)$ and any case of the previous proposition may occur. 

Now, we consider the Grassmann bundle $\psi:\G_{N,N_2,l_2}\to P_{N,N_2,l_2}$ with fiber over a point of $P_{N,N_2,l_2}$ corresponding to a bundle $E$ equal to $G(3,H^0(S,E))$ and denote by $\W_{N,N_2,l_2}$ the closure of the image of the rational map $h_{N,N_2,l_2}:\G_{N,N_2,l_2}\dashrightarrow\W^2_d(\vert L\vert)$. 
\begin{lem}\label{lem:marta}
Assume that general curves in $\vert L\vert$ have Clifford dimension $1$ and maximal gonality $k=\left\lfloor \frac{g+3}{2}\right\rfloor$. Then, for any irreducible component $\W$ of $\W_{N,N_2,l_2}$, one has
\begin{eqnarray*}
\dim\W&\leq& \frac{1}{4}g+d+\frac{3}{2}-\alpha,
\end{eqnarray*}
where $\alpha$ is as in Proposition \ref{prop:sedici}.
\end{lem}
\begin{proof}
Let $\G$ be an irreducible component of $\G_{N,N_2,l_2}$ such that $\W=\overline{h_{N,N_2,l_2}(\G)}$. Since $\G=\psi^{-1}(Z)$ for some irreducible component $Z$ of $P_{N,N_2,l_2}$, 
Proposition \ref{prop:sedici} implies that:
\begin{eqnarray*}
\dim\G&=&3(g-d+2)+\dim Z\\
&=&3(g-d)+12-\alpha-2\chi(E)+2l_1+2l_2+\\
& &c_1(N)\cdot c_1(N_1)+c_1(N)\cdot c_1(N_2)+c_1(N_1)\cdot c_1(N_2)\\
&=&g-d+2-\alpha+2(l_1+l_2)+c_1(N)\cdot (c_1(N_1)+c_1(N_2))+c_1(N_1)\cdot c_1(N_2)\\
&=&g+d+2-\alpha-c_1(N)\cdot c_1(N_1)-c_1(N)\cdot c_1(N_2)-c_1(N_1)\cdot c_1(N_2)\\
&\leq&g+d+2-\alpha-\frac{3}{2}k\\
&\leq&\frac{1}{4}g+d+\frac{1}{2}-\alpha,
\end{eqnarray*}
where we have used Lemma \ref{lem:caso3} and the fact that $k\geq (g+2)/2$. 
The statement follows since the fibers of $h_{N,N_2,l_2}$ are quotient stacks of dimension $-1$.
\end{proof}
Finally, we prove the following:
\begin{prop}\label{prop:nutella}
Assume that general curves in $\vert L\vert$ have Clifford dimension $1$ and maximal gonality $k=\left\lfloor \frac{g+3}{2}\right\rfloor$. Fix a positive integer $d$ such that $(g,d)\not\in\{(2,4),(4,5),(6,6),(10,9)\}$. Let $\W\subset \W_{N,N_2,l_2}$ be an irreducible component of $\W^2_d(\vert L\vert)$. Then $\rho(g,2,d)\geq 0$ and $\W$ does not dominate $\vert L\vert$.
\end{prop} 
\begin{proof}
Lemma \ref{lem:caso3} implies $d\geq\frac{3}{2}k$, hence $\rho(g,2,d)\geq0$. Lemma \ref{lem:marta} gives:
$$
\dim\W\leq \frac{1}{4}g+d+\frac{3}{2}-\alpha.
$$
Therefore, $\W$ cannot dominate $\vert L\vert$ if $$\frac{1}{4}g+d+\frac{3}{2}-\alpha< g+\rho(g,2,d)=-g+3d-6,$$ that is, $d>\frac{5}{8}g+\frac{15}{4}-\frac{\alpha}{2}$. In particular, as $\alpha\geq1$, it is enough to require $d>\frac{5}{8}g+\frac{13}{4}=:h$. Such inequality is satisfied always except for
$$
(g,d)\in\{(2,4),(3,5),(4,5),(5,6),(6,6),(6,7),(8,8),(10,9),(14,12)\}.
$$
If $(g,d)=(6,6)$, the linear system $\vert L\vert$ can be dominated by $\W$.
In all the other cases $d=\lfloor h\rfloor$ and we check whether $\alpha>2h-2\left\lfloor h\right\rfloor+1$, which would prevent $\W$ from being dominant. This holds true if $(g,d)\not\in\{(2,4),(4,5),(10,9)\}$ (use that the case $\alpha=1$ may occur only when parametrizing LM bundles of type (\ref{two}) and that, if $\gcd(2g-2,3)=1$, there do not exist properly $\mu_L$-semistable bundles of Mukai vector $v(E)$). 
\end{proof}
\begin{rem}
The four cases which are not covered by Proposition \ref{prop:nutella} might be treated by  "ad hoc" arguments but this is not the purpose of the paper.
\end{rem}
Proofs of Theorem \ref{thm:magari} and Theorem \ref{thm:principale} are now straightforward.
\begin{proof}[Proof of Theorem \ref{thm:magari}]
Being non-simple, the LM bundle $E_{C,A}$ is not $\mu_L$-stable. Since $d<\frac{2}{3}g+2$, Corollary \ref{cor:lasagna} implies the existence of a line bundle $N\in\Pic(S)$ such that $E_{C,A}\in R_N(\mathbb{C})$. The statement thus follows directly from Proposition \ref{prop:vai}.
\end{proof}

\begin{proof}[Proof of Theorem \ref{thm:principale}]\
Case (\ref{aa}) trivially follows from Proposition \ref{prop:cambridge}, Proposition \ref{prop:fiducia} and Proposition \ref{prop:nutella}. 

Now, let $\frac{2}{3}g+2\leq d\leq \frac{3}{4}g+2$. Given $\W$ an irreducible component of $\W^2_d(\vert L\vert)$ which dominates $\vert L\vert$ and whose general point corresponds to a LM bundle that is not $\mu_L$-stable, Proposition \ref{prop:cambridge} and Proposition \ref{prop:nutella} imply the existence of a line bundle $N\in\Pic(S)$ such that $\W\subset \W_N$. The statement follows from Proposition \ref{prop:vai}.
\end{proof}

\section{Transversality of some Brill-Noether loci}\label{bus}
We apply our results in order to prove Theorem \ref{thm:tra} in the introduction. 
\begin{thm}\label{thm:gon}
Let $r\geq 3$, $g\geq 0$, $d\leq g-1$ be positive integers such that $\rho(g,r,d)<0$ and $d-2r+2\geq \lfloor(g+3)/2\rfloor$. If $r\geq 4$, assume $d^2>4(r-1)(g+r-2)$. For $r=3$, let $d^2>8g+1$. If $-1$ is not represented by the quadratic form
$$
Q(m,n)=(r-1)m^2+mnd+(g-1)n^2\,\,\,m,n\in\mathbb{Z},
$$
there exists a smooth curve $C\subset \mathbb{P}^r$ of genus $g$, degree $d$ and maximal gonality $\left\lfloor \frac{g+3}{2}\right\rfloor$. Moreover, one can choose $C$ such that for any complete, base point free $g^1_{e}$ on $C$ with $\rho(g,1,e)\geq 0$ the Petri map is injective.
\end{thm}
\begin{proof}
Notice that the inequalities $d\leq g-1$ and $d^2>4(r-1)(g-1)$ trivially imply $d>4(r-1)$.

In order to prove the first part of the statement, we proceed as in \cite{gabi1} Theorem 3 paying special attention to our slightly different hypotheses. Rathmann's Theorem implies the existence of a $2r-2$-degree $K3$ surface $S\subset \mathbb{P}^r$ and a smooth curve $C\subset S$ of degree $d$ and genus $g$ such that $\Pic(S)=\mathbb{Z}H\oplus\mathbb{Z}C$, where $H$ is the hyperplane section of $S$. Our assumption on $Q$ implies that $S$ does not contain $(-2)$-curves. As in \cite{gabi1}, one shows that the line bundle $L:=\oo_S(C)$ is ample by Nakai-Moishezon criterion (if $D\subset S$ is an effective divisor, use that $D^2\geq 0$ and $D\cdot H>2$, in order to show that $C\cdot D>0$). Hence, $C$ has Clifford dimension $1$ (cf. \cite{ciliberto} Proposition 3.3).

Assume that $C$ has gonality $k<\left\lfloor \frac{g+3}{2}\right\rfloor$. We reach a contradiction by showing that $k\geq d-2r+2$. If $A$ is a complete, base point free pencil $g^1_k$ on $C$, by \cite{donagi} Theorem (4.2) there exists an effective divisor $D\equiv mH+nC$ on $S$, such that $\vert A\vert$ is contained in the linear system $\vert \oo_C(D)\vert$ and the following conditions are satisfied:
$$
h^0(S,\oo_S(D))\geq 2,\,\,h^0(S,\oo_S(C-D))\geq 2,\,\,C\cdot D\leq g-1,\,\, \Cliff(D\vert_C)= \Cliff(A).$$
In particular, as remarked in \cite{donagi} page 60, the last equality implies that $$h^1(S,\oo_S(D))=h^1(S,\oo_S(C-D))=0,$$ thus $c_1(D)^2>0$ and $c_1(C-D)^2>0$. Moreover, one has $$k=2+\Cliff(D\vert_C)=D\cdot (C-D).$$ We show that
$$ 
f(m,n)=D\cdot C-D^2=-(2r-2)m^2+d(1-2n)m+(n-n^2)(2g-2)\geq d-2r+2,
$$
for values of $m$ and $n$ satisfying the following inequalities:
\begin{enumerate}
\item\label{nuovo1} $(r-1)m^2+mnd+n^2(g-1)>0$,
\item\label{nein} $(r-1)m^2+(mn-m)d+(1-n)^2(g-1)>0$,
\item\label{nuovo2} $2<(2r-2)m+nd<d-2$,
\item\label{nuovo3} $md+(2n-1)(g-1)\leq 0$.
\end{enumerate}

Assume first that $n=1$, and set $a = -m$. Then (\ref{nuovo2}) implies $0 < a < (d-2)/(2r-2)$. Inequality (\ref{nuovo1}) is equivalent to $(r-1)a^2-ad+g-2\geq 0$, whence $$a\leq\frac{d-\sqrt{d^2-4(r-1)(g-2)}}{2r-2}.$$ We have $f(-a,1)\geq d-2r+2$ whenever $1\leq a\leq d/(2r-2)-1$. For either $r\geq 4$ or $r=3$ and $d^2-8g\geq8$, this holds true because $d^2-4(r-1)(g-2)\geq4r(r-1)>4(r-1)^2$. If $r=3$ and $d^2-8g<8$, then $d^2-8g=4$ and $d\equiv2\mod4$. Hence, (\ref{nuovo2}) implies that $1\leq a< (d-4)/4$. Remark that $f(-a,1)=d-2r+2$ whenever $a=1$, that is, $C\equiv C-H$.
The case $n=0$ can be treated similarly by using (\ref{nein}) instead of (\ref{nuovo1}), and one obtains that $f(m,0)\geq d-2r+2$ with equality holding only for $m=1$, that is, $D\equiv H$.

If $n<0$, inequalities (\ref{nuovo1}), (\ref{nuovo2}) and (\ref{nuovo3}) imply that $-\alpha n<m\leq(g-1)(1-2n)/d$, where $$\alpha=\frac{d+\sqrt{d^2-4(r-1)(g-1)}}{2r-2}.$$ Therefore, one has
$$
f(m,n)\geq \min\left\{f(-\alpha n,n),f\left(\frac{(g-1)(1-2n)}{d},n\right)\right\}.
$$
Analogously, if $n\geq 2$, then $\max\{-\beta n, (2-nd)/(2r-2)\}<m\leq(g-1)(1-2n)/d$, where
$$
\beta=\frac{d-\sqrt{d^2-4(r-1)(g-1)}}{2r-2};
$$ 
this gives
$$
f(m,n)\geq \min\left\{f\left(\frac{(g-1)(1-2n)}{d},n\right), \max\left\{f(-\beta n,n),f\left(\frac{2-nd}{2r-2},n\right)\right\}\right\}.
$$
Computations in \cite{gabi1} give $\max\left\{f(-\beta n,n),f\left((2-nd)/(2r-2),n\right)\right\}>d-2r+2$ if $n\geq 2$, and $f(-\alpha n,n)>d-2r+2$ when $n<0$, unless $r=3$, $n=-1$ and $d^2-8g=4$. In this case, $d\equiv 2\mod 4$ and $m\geq (d+4)/4$ by (\ref{nuovo2}); one uses that $f((d+4)/4,-1)>d-4$.
In order to conclude the proof that $C$ has maximal gonality, it is enough to remark that the function $$h(n):=f\left(\frac{(g-1)(1-2n)}{d},n\right)=\frac{g-1}{2}\left[\frac{(2n-1)^2(d^2-4(r-1)(g-1))}{d^2}+1\right]$$ reaches its minimum for $n=1/2$ and $h(0)\geq d-2r+2$ by direct computation.

Concerning the last part of the statement, assume $C$ is general in its linear system and let $A$ be a complete, base point free pencil $g^1_{e}$ on $C$ such that $\rho(g,1,e)\geq 0$ and $\ker\mu_{0,A}\neq 0$. The bundle $E=E_{C,A}$ is non-simple, hence it cannot be $\mu_L$-stable. As a consequence, there exists a short exact sequence
\begin{equation}\label{fu}
0\to M\to E\to N\otimes I_{\xi}\to 0,
\end{equation}
where $M,N$ are line bundles, $I_{\xi}$ is the ideal sheaf of a $0$-dimensional subscheme $\xi\subset S$ and $c_1(M)\cdot C\geq \mu_L(E)=g-1\geq c_1(N)\cdot C$. If sequence (\ref{fu}) does not split, then
$$
h^0(S,E\otimes E^\vee)\leq 1+\dim\Hom(M,N\otimes I_\xi)+\dim\Hom(N\otimes I_\xi,M).
$$
Since $\mu_L(M)\geq \mu_L(N)$, if $\Hom(M,N\otimes I_\xi)\neq 0$ then $M\simeq N$ and $C=2c_1(M)$, which is absurd. It follows that $N^\vee\otimes M$ is non-trivial and effective. Since $S$ does not contain $(-2)$-curves, one has 
$$
c_1(N^\vee\otimes M)^2=C^2-4c_1(N)\cdot c_1(M)=2g-2-4c_1(N)\cdot c_1(M)\geq 0;$$
this contradicts Lemma \ref{lem:uova}, which states that $c_1(N)\cdot c_1(M)\geq k\geq(g+2)/2$. Thus, $\xi=\emptyset$ and sequence (\ref{fu}) splits. We have to show that, if $E=N\oplus M$ is a splitting LM bundle, the rational map $\chi:G(2,H^0(S,E))\dashrightarrow \vert L\vert$ cannot be dominant. Remark that $\chi$ factors through the rational map $h_E:G(2,H^0(S,E))\dashrightarrow \W^1_e(\vert L\vert)$, whose fiber over a point $(C,A)\in\im\, h_E$ is at least $1$-dimensional since it is isomorphic to $\mathbb{P}(\Hom(E_{C,A},\omega_C\otimes A^\vee)^\circ)$, where $\Hom(E_{C,A},\omega_C\otimes A^\vee)^\circ$ is an open subgroup of $\Hom(E_{C,A},\omega_C\otimes A^\vee)\simeq H^0(S,E_{C,A}\otimes E_{C,A}^\vee)$. This is enough to conclude because $\rho(g,1,e)\geq 0$, hence $\dim G(2,H^0(S,E))=2(g-e+1)\leq g$.
\end{proof}

\begin{thm}\label{thm:pappa}
Let $g,r,d$ satisfy the hypotheses of Theorem \ref{thm:gon}. The curve $C$ can be chosen such that, if $$e< \min\left\{d-2r+5,\frac{17}{24}g+\frac{23}{12}\right\},$$ then $C$ does not have any complete, base point free net $g^2_{e}$ for which the Petri map is non-injective.
\end{thm}
\begin{proof}
Let $S\subset \mathbb{P}^r$ be as in the proof of Theorem \ref{thm:gon} and $C$ be general in its linear system. Let $A$ be a complete, base point free net on $C$ of degree $d_A< \frac{17}{24}g+\frac{23}{12}$; if $\rho(g,2,d_A)\geq 0$, assume moreover that $\ker\mu_{0,A}\neq 0$. Corollary \ref{cor:lasagna} and Proposition \ref{prop:vai} imply that $\vert A\vert$ is contained in the linear system $\vert\oo_C(D)\vert$ for some effective divisor $D\equiv mH+nC$ on $S$ such that:
$$
h^0(S,\oo_S(D))\geq 2,\,\,h^0(S,\oo_S(C-D))\geq 2,\,\,C\cdot D\leq \frac{4g-4}{3},\,\, \Cliff(D\vert_C)\leq \Cliff(A).$$
In fact, the Lazarsfeld-Mukai bundle $E:=E_{C,A}$ is given by an extension:
$$
0\to N\to E\to E/N\to 0,
$$
where $N:=\oo_S(C-D)$ and $E/N$ is a $\mu_L$-stable torsion free sheaf of rank $2$ on $S$. As in the proof of Proposition \ref{prop:vai}, one shows that $D^2>0$, hence $h^1(S,\oo_S(D))=0$. Moreover, one obtains that $h^1(S,N)=0$ because the equality $(C-D)^2=0$ would imply $d\geq (5g+4)/6$, which is absurd. 
As a consequence, one has
\begin{equation}\label{cara}
d_A-4 =\Cliff(A)\geq \Cliff(D\vert_C)=D\cdot C-2h^0(S,\oo_C(D\vert_C))+2=D\cdot (C-D)-2,
\end{equation}
and equality holds only if $D^2=2$ and $c_2(E/N)=2$ (cf. proof of Proposition \ref{prop:vai}); in particular, for $D\equiv H$, the inequality is strict. We show that 
\begin{equation}\label{semi}
f(m,n):=D\cdot (C-D)\geq d-2r+2,
\end{equation}
and, if equality holds, then either $D\equiv H$ or $D\equiv C-H$. Computations are similar to those in Theorem \ref{thm:gon}, but now, instead of having $D\cdot C\leq g-1$, we only know that $D\cdot C\leq(4g-4)/3$. Therefore, inequality (\ref{nuovo3}) must be replaced with
\begin{itemize} 
\item[(iv')]\label{bla} $md+(2n-\frac{4}{3})(g-1)\leq 0$.
\end{itemize}
The cases $n\in\{0,1\}$ can be treated exactly as before. For $n<0$, we have 
$$
f(m,n)\geq \min\left\{f(-\alpha n,n),f\left(\frac{(g-1)(\frac{4}{3}-2n)}{d},n\right)\right\}.
$$
If $n\geq 2$, then
$$
f(m,n)\geq \min\left\{f\left(\frac{(g-1)(\frac{4}{3}-2n)}{d},n\right), \max\left\{f(-\beta n,n),f\left(\frac{2-nd}{2r-2},n\right)\right\}\right\}.
$$
 Therefore, it is enough to show that 
$$
g(n):=f\left(\frac{(g-1)(\frac{4}{3}-2n)}{d},n\right)-d+2r-2> 0\textrm{ for }n< 0 \textrm{ or }n\geq 2.
$$
One can write $g(n)=an^2+bn+c$, with
\begin{eqnarray*}
a&=&-4(2r-2)\left(\frac{g-1}{d}\right)^2+2d\left(\frac{g-1}{d}\right),\\
b&=&\frac{16}{3}(2r-2)\left(\frac{g-1}{d}\right)^2-\frac{8}{3}d\left(\frac{g-1}{d}\right),\\
c&=&-\frac{16}{9}(2r-2)\left(\frac{g-1}{d}\right)^2+\frac{4}{3}d\left(\frac{g-1}{d}\right)-d+2r-2.
\end{eqnarray*}
Since $a>0$ and $0<-b/(2a)<1$, our claim follows if $g(0)=c>0$, or equivalently, if 
$$
\frac{3}{4}<\frac{g-1}{d}<\frac{3}{8}\left(\frac{d-2(r-1)}{r-1}\right).
$$
The left inequality is trivial since $d\leq g-1$. The right inequality is equivalent to the condition $8(g-1)(r-1)<3d^2-6d(r-1)$, which is satisfied as well (if $r\geq 4$, use that $8(g-1)(r-1)<2d^2-8(r-1)^2$ and $d> 4(r-1)$; if $r=3$, use that $d^2>8g+1$ and either $(g,d)=(12,11)$ or $d\geq 12$ by manipulation of the hypotheses). 

We conclude that $d_A\geq d-2r+4$ and the inequality is strict unless equalities hold both in (\ref{cara}) and (\ref{semi}), thus $D\equiv C-H$ and $(C-H)^2=2$. This case can be excluded since it would imply $d=g+r-3\geq g$.
\end{proof}
Remark that the condition $e<\frac{17}{24}g+\frac{23}{12}$ is automatically satisfied if $\rho(g,2,e)<0$.

The proof of Theorem \ref{thm:tra} is now trivial: apply Theorem \ref{thm:gon} and Theorem \ref{thm:pappa} and proceed by induction on $f$ and $e$ in order to deal with pencils $g^1_f$ and nets $g^2_e$ which have a nonempty base locus.

\section{Noether-Lefschetz divisor and Gieseker-Petri divisor in genus 11}
The Clifford index $\Cliff(C)$ is one of the most important invariants of an algebraic curve $C$. In \cite{lange} Lange and Newstead defined the analogue of the Clifford index for higher rank vector bundles in the following way. If $\U_C(n,d)$ denotes the moduli space of semistable rank-$n$ vector bundles of degree $d$ on a genus-$g$ curve $C$, given $E\in \U_C(n,d)$, the Clifford index of $E$ is
$$
\gamma(E):=\mu(E)-\frac{2}{n}h^0(C,E)+2\geq 0,
$$
where $\mu(E)$ denotes the slope of $E$. For any positive integer $n$, one defines the higher Clifford index of $C$ 
$$
\Cliff_n(C):=\mathrm{min}\{\gamma(E)\,\vert\, E\in \U_C(n,d),\,h^0(C,E)\geq 2n,\,\mu(E)\leq g-1\}.
$$
A natural question is whether higher Clifford indices are new invariants, different from the ones arising in classical Brill-Noether theory. In \cite{lange} Lange and Newstead reformulated a conjecture of Mercat (cf. \cite{mercat}) in a slightly weaker form predicting:
\begin{equation}\label{rugby}
\Cliff_n(C)=\Cliff(C);
\end{equation}
remark that trivially $\Cliff_n(C)\leq \Cliff(C)$, while the opposite inequality is largely non-trivial.
When $n=2$, the conjecture has been proved for a general curve in $M_g$ if $g\leq 16$ by Farkas and Ortega (cf. \cite{angela}) and the same is expected to hold true in any genus. However, if $g\geq 11$, there are examples of curves with maximal Clifford index $\Cliff(C)=\left\lfloor\frac{g-1}{2}\right\rfloor$ that violate (\ref{rugby}) for $n=2$. These have been constructed in \cite{angela}, \cite{gan}, \cite{lange}, \cite{peter}, \cite{new} as sections of $K3$ surfaces with Picard number at least $2$. We recall that the $K3$ locus 
$$
\K_g:=\{[C]\in M_g\,\vert\, C\subset S,\,S\textrm{ is a }K3\textrm{ surface}\}
$$
is irreducible of dimension $19+g$ if either $g=11$ or $g\geq 13$ (cf. \cite{cil}). In particular, $\K_{11}=M_{11}$ and a general curve $[C]\in M_{11}$ lies on a unique $K3$ surface with Picard number $1$ (cf \cite{muk}). Given two positive integers $r,d$ such that $d^2>4(r-1)g$ and $d$ does not divide $2r-2$, one defines the Noether-Lefschetz divisor inside $\K_g$ as
$$
\NL^r_{g,d} := \left\{ [C] \in \K_g \, \left\vert\begin{array}{l}  C \subset S, \, S\textrm{ is a } K3 \textrm{ surface}, \,\Pic(S) \supset \Z  C \oplus \Z  H, \\
 H \textrm{ nef }, \, H^2 = 2r - 2, \, C^2 = 2g - 2, \, C\cdot H = d  \end{array}\right.\right\}.$$
In \cite{gan} it is proved that a curve $C$ of genus $11$ violates Mercat's conjecture for $n=2$ whenever $[C]\in\NL^4_{11,13}$. 

Since some of the curves exhibited in \cite{lange}, \cite{peter}, \cite{new} do not satisfy the Gieseker-Petri Theorem, Lange and Newstead asked whether $\Cliff_2(C)=\Cliff(C)$ whenever $C$ is a Petri curve (Question 4.2 in \cite{new}). We prove Theorem \ref{thm:chisa}, which gives a negative answer to this question.

Let $S\subset \mathbb{P}^4$ be a $K3$ surface such that $\Pic(S)=\Z C\oplus\Z H$, where $H$ is the hyperplane section, $H^2=6$, $C^2=20$ and $C\cdot H=13$. Denote by $L$ the line bundle $\oo_S(C)$. We show that, if $C\in\vert L\vert$ is general, then $[C]$ does not lie in the Gieseker-Petri locus $GP_{11}$. Recall that $GP_{11}$ has pure codimension $1$ in $M_{11}$ (cf. \cite{marghem}) and decomposes in the following way:
$$
GP_{11}=M^2_{11,9}\cup GP^2_{11,10}\cup\bigcup_{d=7}^{10} GP^1_{11, d},
$$
where $M^2_{11,9}$ is a Brill-Noether divisor. Therefore, proving the transversality of $\NL^4_{11,13}$ and $GP_{11}$ is equivalent to showing that in the above situation, if $C\in\vert L\vert$ is general, then $C$ has no $g^2_9$ and the varieties $G^2_{10}(C)$ and $G^1_d(C)$ for $7\leq d\leq 10$ are smooth of the expected dimension. 

We proceed as in the previous section; since the hypotheses of Theorem \ref{thm:gon} are not satisfied, explicit computations must be performed. Direct calculations imply that $S$ does not contain any $(-2)$-curve. Moreover, $C$ is an ample line bundle on $S$ by Proposition 2.1 in \cite{new}. As a consequence, $C$ has Clifford dimension $1$ (cf. Proposition 3.3 in \cite{ciliberto}) and $\Cliff(C)=5$ (cf. Proposition 3.3 in \cite{gan}). In particular, $C$ has maximal gonality $k=7$ and has no $g^2_d$ for $d\leq 8$. Hence, in order to prove that $G^2_9(C)= \emptyset$, it is enough to exclude the existence of complete, base point free $g^2_{9}$ on $C$. Similarly, the condition $[C]\not\in GP^2_{11,10}$ is equivalent to the requirement for $G^2_{10}(C)$ to be smooth of the expected dimension $\rho(11,2,10)$ in the points corresponding to complete, base point free linear series. Analogously, by induction on $d$, if the Petri map associated with any complete, base point free pencil of degree $7\leq d\leq 10$ is injective, then $[C]\not\in \cup_{d=7}^{10}GP^1_{11,d}$. 

For any $A\in G^2_9(C)$, the Petri map $\mu_{0,A}$ is non-injective for dimension reasons and the bundle $E=E_{C,A}$ is non-simple, hence it cannot be $\mu_L$-stable.  Since $$\gcd(rk\, E, c_1(E)^2)=\gcd(3,20)=1,$$ there are no properly semistable sheaves of Mukai vector $v(E)=(3,C,4)$; hence, $E$ is $\mu_L$-unstable. By Corollary \ref{cor:lasagna}, $E$ sits in the short exact sequence
\begin{equation}\label{panda}
0\to N\to E\to E/N\to 0,
\end{equation}
where $N\in\Pic(S)$ is its maximal destabilizing sheaf and the quotient $E/N$ is a $\mu_L$-stable torsion free sheaf of rank $2$. Having denoted by $D$ the first Chern class of $E/N$, Proposition \ref{prop:vai} implies that the line bundle $\oo_C(D)$ contributes to $\Cliff(C)$. Moreover, as in the proof of the aforementioned proposition, one shows that 
\begin{eqnarray}
\label{dis1}D^2&\geq& 2,\\
\label{dis2}c_2(E/N)&\geq &\frac{3}{2}+\frac{1}{4}D^2.
\end{eqnarray}
Furthermore, Lemma \ref{lem:endo1} gives
\begin{equation}\label{dis3}
c_1(N)\cdot c_1(E/N)=(C-D)\cdot D\geq k=7.
\end{equation}
Since 
$$
9=c_2(E)=c_2(E/N)+(C-D)\cdot D\geq \frac{3}{2}+\frac{1}{4}D^2+(C-D)\cdot D,
$$
the divisor $D\equiv mH+nC$ must satisfy
$$
\left\{\begin{array}{l}C\cdot D=13m+20n=9\\D^2=6m^2+20n^2+26mn=2(m+n)(3m+10n)=2\end{array}\right..
$$
One shows that this system admits no integral solution. As a consequence, a general curve in $\vert L\vert$ has no linear series of type $g^2_9$.\vspace{0.4cm}

Analogously, given a complete, base point free $A\in G^2_{10}(C)$ with $\ker\mu_{0,A}\neq 0$, the LM bundle $E=E_{C,A}$ is $\mu_L$-unstable and its maximal destabilizing sheaf is a line bundle $N$ such that $E/N$ is $\mu_L$-stable by Corollary \ref{cor:lasagna}. With the same notation as above, inequalities (\ref{dis1}), (\ref{dis2}), (\ref{dis3}) still hold true and the following cases must be considered:
$$
\begin{aligned}
(a)\,\,\left\{\begin{array}{l}C\cdot D=10\\D^2=2\\(c_2(E/N)=2)\end{array}\right.\,\,\,\,\,\,\,\,&
(b)\,\,\left\{\begin{array}{l}C\cdot D=9\\D^2=2\\ (c_2(E/N)=3)\end{array}\right.\\
(c)\,\,\left\{\begin{array}{l}C\cdot D=11\\D^2=4\\ (c_2(E/N)=3)\end{array}\right.\,\,\,\,\,\,\,\,&
(d)\,\,\left\{\begin{array}{l}C\cdot D=13\\D^2=6\\ (c_2(E/N)=3)\end{array}\right..\\
\end{aligned}
$$
These systems have no integral solutions except for (d), which is satisfied by $$(m,n)=(1,0).$$ Therefore, $N=\oo_S(C-H)$ and $v(E/N)=(2,H,2)$. Since $\langle v(E/N),v(E/N)\rangle=-2$, the sheaf $E/N$ is uniquely determined. 

By applying first $\Hom(E,-)$ and then $\Hom(-,N)$ and $\Hom(-,E/N)$ to the short exact sequence (\ref{panda}), one shows that
$$
h^0(S,E\otimes E^\vee)\leq 2+\dim\Hom(N, E/N)+\dim\Hom(E/N,N)$$
and the inequality is strict if the sequence does not split. Since $\mu_L(N)>\mu_L(E/N)$, Proposition \ref{prop:morfismi} implies that $\Hom(N,E/N)=0$. Let $0\neq \alpha\in\Hom(E/N,N)$. Since both $\im\,\alpha$ and $\ker\alpha$ are torsion free sheaves of rank $1$, there exists an effective divisor $D_1$ on $S$ and two $0$-dimensional subschemes $\xi_1,\xi_2\subset S$ such that $E/N$ is given by an extension
$$
0\to\oo_S(2H-C+D_1)\otimes I_{\xi_1}\to E/N\to \oo_S(C-H-D_1)\otimes I_{\xi_2}\to 0.$$
The $\mu_L$-stability of $E/N$ implies that $$13/2=\mu_L(E/N)<(C-H-D_1)\cdot C=-D_1\cdot C+7;$$ since $C$ has positive intersection with any non-trivial effective divisor, $D_1=0$. It follows that
$$
3=c_2(E/N)=(2H-C)\cdot (C-H)+l(\xi_1)+l(\xi_2)\geq 7,
$$
which is absurd. Hence, $\Hom(E/N,N)=0$ and (\ref{panda}) splits. As a consequence, the bundle $E=N\oplus E/N$ is uniquely determined. 

We look at the rational map $\chi:G(3,H^0(S,E))\dashrightarrow \vert L\vert$; this cannot be dominant since $\dim G(3,H^0(S,E))=9$. Therefore, a general curve $C\in\vert L\vert$ does not lie in $GP^2_{11,10}$.\vspace{0.4cm}

It remains to show that, if $C\in\vert L\vert$ is general, then $[C]\not\in \cup_{d=7}^{10}GP^1_{11,d}$. It is enough to prove that for any complete, base point free $g^1_d$ on $C$ the Petri map is injective. One can proceed exactly as in the last part of the proof of Theorem \ref{thm:gon} since $S$ does not contain $(-2)$-curves.

\end{document}